\documentclass[10pt,a4paper,reqno]{amsart}
\usepackage[utf8]{inputenc}
\usepackage[T1]{fontenc}
\usepackage{amsmath}
\usepackage{amsthm}
\usepackage{amssymb}
\usepackage{graphicx}
\usepackage[english]{babel}
\usepackage{xfrac}
\usepackage[hidelinks]{hyperref}
\usepackage{mathtools}
\usepackage{xcolor}
\usepackage{soul}
\usepackage{geometry}
\newtheorem{theorem}{Theorem}[section] 
\newtheorem{lemma}[theorem]{Lemma}

\newtheorem{prop}[theorem]{Proposition}
\newtheorem{corollary}{Corollary}[section]
\theoremstyle{remark}
\newtheorem{remark}{Remark}[section]

\newcommand{\C}{\mathbb{C}}

\newcommand{\R}{\mathbb{R}}
\newcommand{\I}{\mathbb{I}}
\newcommand{\gen}[1]{\langle #1 \rangle}
\newcommand{\supp}{\text{supp}}
\newcommand{\nablaA}{\nabla_{\!\text{A}}}

\numberwithin{equation}{section}

\newcommand{\myfootnote}[1]{
	\renewcommand{\thefootnote}{}
	\footnotetext{\scriptsize#1}
	\renewcommand{\thefootnote}{\arabic{footnote}}
}

\title{On the point spectrum of electromagnetic Dirac operators}	

\author{Naiara Arrizabalaga} 
\author{Lucrezia Cossetti}
\author{Matias Morales}

\date{\today}

\begin{document}
	
	\myfootnote{Naiara Arrizabalaga: Universidad del Pa\'is Vasco/Euskal Herriko Unibertsitatea, UPV/EHU, Aptdo. 644, 48080, Bilbao, Spain;\\
		email: \href{naiara.arrizabalaga@ehu.eus}{naiara.arrizabalaga@ehu.eus}}
	\myfootnote{Lucrezia Cossetti: Universidad del Pa\'is Vasco/Euskal Herriko Unibertsitatea, UPV/EHU, Aptdo. 644, 48080, Bilbao, Spain;\\
		email: \href{lucrezia.cossetti@ehu.eus}{lucrezia.cossetti@ehu.eus}}
	\myfootnote{Matias Morales: Universidad del Pa\'is Vasco/Euskal Herriko Unibertsitatea, UPV/EHU, Aptdo. 644, 48080, Bilbao, Spain;\\
		email: \href{matiasbenjamin.morales@ehu.eus}{matiasbenjamin.morales@ehu.eus}}

	\begin{abstract}
		In this work, we develop the method of multipliers for electromagnetic Dirac operators and establish sufficient conditions on the magnetic and electric fields that guarantee the absence of point spectrum.
		In the massless case, our approach covers Coulomb-type potentials of the form $V(x)=\frac{1}{|x|} \big(\nu \I + \mu \beta + i \delta \beta \big(\boldsymbol{\alpha}\cdot \frac{x}{|x|} \big) \big).$ We also adapt the method to show absence of embedded eigenvalues above a threshold which depends on the asymptotic behaviour of the magnetic and electric fields. 
	\end{abstract}
	\maketitle
	
	\section{Introduction}\label{section: introduction}
	In this work we are interested in spectral properties of the \emph{electromagnetic} Dirac operator
	\begin{equation}\label{eq:electromagnetic-Dirac} 
		H_m(A,V) = -i\boldsymbol{\alpha} \cdot \nabla_{\!\text{A}} +m\beta+V.
	\end{equation}
	Here  $\boldsymbol{\alpha} = (\alpha_1,\alpha_2, \alpha_3),$ with $\alpha_j,$ $j=1,2,3$ and $\beta$ being the standard $4\times 4$ Hermitian Dirac matrices, $m\geq 0$ denotes the mass of the particle, and $V$ is a potential function that associates to each $x\in \R^3$ a $4\times 4$ Hermitian matrix $V(x).$ Moreover, $\nablaA$ stands for the magnetic gradient $\nabla-iA,$ where $A\colon \R^3 \to \R^3$ is  a magnetic vector potential.
	The operator~\eqref{eq:electromagnetic-Dirac} represents the Hamiltonian of a relativistic particle of spin-$\tfrac{1}{2}$ under the action of an \emph{electric} field $V$ and a \emph{magnetic} field $B = \nabla \times A.$
	
	Compared to the cases of purely electric~\cite{Cuenin17,Fanelli2019, Kalf1976,Kalf2003,Okaji2003,Schmidt2010,Schmidt2015} or purely magnetic~\cite{Balinsky2001,Cossetti2020,Frank2022a,Frank2022,Hundertmark2024} perturbations, the literature available related to the absence of point spectrum of \emph{electromagnetic} Dirac operators is much scarcer and much less complete. The results presented in this work are motivated by the need to fill in this gap. One of the first works available on this topic is due to Roze~\cite{Roze1970}, who considered electric potentials of the form  $V(x) = v(x)\I,$ and proved that if $A$ and $v$ are smooth and satisfy 
	\begin{equation}\label{eq:Roze} 
		\lim_{|x| \to \infty}|x|(|A(x)|+|v(x)|) = 0, 
	\end{equation}
	then the Dirac operator has no eigenvalues embedded in the continuous spectrum. A similar result was proved by Kalf in~\cite{Kalf1981}, where he replaced the decay assumption~\eqref{eq:Roze} by an analogous assumption involving the magnetic field $B$ instead of the magnetic potential $A$, thus giving a gauge invariant condition for the absence of eigenvalues. The result of Kalf was later extended to more general matrix-valued potentials in~\cite{Berthier1987}. In a more recent result~\cite{Richard2007}, the authors investigated on spectral properties of Dirac operators with magnetic fields of constant direction and Coulomb-type electric potentials. 
	
	More recently, Cossetti, Fanelli and Krej\v{c}i\v{r}\'\i k \cite{Cossetti2020} have proved the total absence of eigenvalues for purely magnetic Dirac operators in dimension $d \geq 3$ if the following inequality \[ 
	\int |x|^2|B|^2|\psi| \leq c \int|\nablaA \psi|^2, 
	 \] 
	holds for every $\psi \in C_0^\infty(\R^d)$ and for $c > 0$ small enough, that is, if the magnetic field is small in a Hardy sense. In addition, Hundertmark and Kova\v{r}\'ik \cite{Hundertmark2024}, using similar methods, proved the absence of \emph{embedded} eigenvalues under some weaker hypothesis (See Remark \ref{remark teo 2}). These works rely heavily on the supersymmetric structure of the purely magnetic Dirac operator, a property that no longer holds for more general electromagnetic perturbations.
	
	In the present paper, in order to deal with such general perturbations, we will develop the so-called \emph{method of multipliers} and establish a source of sufficient conditions (decay/Hardy-type) on the potentials to guarantee the absence of point spectrum (including possibly embedded eigenvalues) of the Dirac Hamiltonian $H_m(A,V)$ in dimension $d\geq 3.$

	In the last decade, starting from the pioneering work~\cite{Fanelli2018a}, this method was shown to be a very powerful tool in spectral theory and, in particular, in proving absence of eigenvalues (embedded and discrete) of very different Hamiltonians~\cite{Hundertmark2020,Cossetti2020,Cossetti2024a,Cossetti2024, Cossetti2020a, Fanelli2018, Fanelli2018a, Hundertmark2024} (see also~\cite{Cossetti2024} for a survey on the method of multipliers in spectral theory). Despite its robustness, this method does not seem to apply \emph{directly} to general Dirac operators. Indeed, several technical difficulties arise from the lack of positivity of certain commutators, a feature closely connected to the unboundedness from below of the operator itself and to the lack of convexity of the observable position (\emph{Zitterbewegung}). A natural strategy to recover such positivity is to employ the well-known supersymmetric structure of the free Dirac operator, $H_m(0,0)=-i\boldsymbol{\alpha} \cdot \nabla +m\beta$. This means that squaring out the operator gives 
	\begin{equation*}
		H_m(0,0)^2 = (-\Delta +m^2)\I_{\C^4}.
	\end{equation*}
	An important consequence of this spectral relation is that proving absence of eigenvalues of $H_m(0,0)$ reduces to prove  absence of eigenvalues of second order operators of Schrödinger type for which the aforementioned positivity is guaranteed. It is easy to see that the supersymmetric property is maintained for a \emph{purely magnetic} Dirac operator, namely for $H_m(A,0)$:
	\begin{equation*}
		H_m(A,0)^2=
		\begin{pmatrix}
			H_P + m^2 \I_{\C^2} & 0\\
			0 & H_P + m^2 \I_{\C^2}
		\end{pmatrix},		
	\end{equation*} 
	where $H_P:=-\Delta \I_{\C^2} + \boldsymbol{\sigma} \cdot B$ denotes the Pauli operator in $L^2(\R^3)^2.$ Here $\sigma=(\sigma_1,\sigma_2,\sigma_3)$ are the standard $2\times 2$ Hermitian Pauli matrices (see for instance~\cite{Cossetti2024} and~\cite{Hundertmark2024}). However, introducing a general electric potential $V$ breaks this property, making the analysis of such operators much harder. As a consequence, related results have so far been obtained only through alternative techniques; see, for instance,~\cite{DAncona2022,DAncona2022a,Dolbeault2024,Fanelli2019, Hansmann2022, Mizutani2022}.
	Overcoming these limitations is precisely the aim of the present work.

We shall now present the main results of this paper. The first one shows that, if the potential satisfies a suitable Hardy-type smallness condition, no eigenvalues are created in the spectral gap.
	\begin{theorem}\label{teo 1}
		Let $d\geq 1$ and let $V: \R^d \to \C^{N \times N},$ $N:= 2^{\lfloor(d+1)/2 \rfloor},$ be such that the following inequality 
		\begin{equation}\label{eq: hipo teo 1}
			\|V\psi \| \leq \|H_m(A)\psi\|-m\|\psi\| 
		\end{equation}
		holds for every $\psi \in H^1(\R^d)$. Then $\sigma_{p}(H_m(A,V))\cap (-m,m) = \varnothing$. 
	\end{theorem}

	\begin{remark}
		It follows from an easy computation that
		\[
		\|H_m(A)\psi\|-m\|\psi\| \leq \|-i(\alpha \cdot \nablaA)\psi\|.
		\] 
		Thus, condition~\eqref{eq: hipo teo 1} can be interpreted as a requirement for the potential $V$ to satisfy a magnetic Hardy inequality. As with Schrödinger-type operators, this kind of condition is not usually sufficient to guarantee the total absence of eigenvalues; it only ensures the absence of discrete eigenvalues.
	\end{remark}
	The main result of this paper reads as follows.
	\begin{theorem}\label{teo: general Dirac}
		Let $d\geq 3.$ Suppose that $V: \R^d \to \C^{N\times N}$ is such that $V \in C^\infty(\R^d\setminus\{0\})$ and that $V$ is bounded outside any ball containing the origin. Additionally, assume that there exist constants $\varepsilon_1, \varepsilon_2, \varepsilon_3,\varepsilon_4 \geq 0$ satisfying 
		\begin{align}\label{eq: hipotesis pequeñez}
			\left(\frac{4d-6}{d-2}\right)\varepsilon_1
			+\left(\frac{4}{d-2}\right)\varepsilon_2\,\varepsilon_3
			+2m\varepsilon_4^2
			+\left(\frac{8d-8}{d-2}\right)\varepsilon_3
			+\left(\frac{4d-4}{d-2}\right)\varepsilon_2 \leq 2,
		\end{align}
		and assume that  
		\begin{align}
			\int |x|^2|B|^2|\psi|^2 &\leq \varepsilon_1^2 \|\nablaA \psi\|^2, \label{eq: B}\\
			\int |x|^2|\nabla V|^2|\psi|^2 & \leq \varepsilon_2^2 \|\nablaA \psi\|^2, 
			\label{eq: hipotesis 1}\\
			\sup_{x \in \R^d}|x||V(x)| &\leq \varepsilon_3, \label{eq: hipotesis sup}
		\end{align}
		and, if $m\neq 0$ or $\{\beta, V\}\neq 0,$ assume further that 
		\begin{equation}
			\label{eq: condition mass}
			\int |x||\nabla V| |\psi|^2 
			\leq \varepsilon_4^2 \|\nablaA \psi\|^2
		\end{equation}
		holds true for every $\psi \in H^1(\R^d)$. Then $\sigma_p(H_m(A,V)) = \varnothing$.
	\end{theorem}
	
	\medskip
	\begin{remark}	
	Theorem~\ref{teo: general Dirac} remains valid if we move the singularity of the potentials from the origin to another point in $\R^d$. 
	\end{remark}
	\begin{corollary}\label{cor:1}
		Let $d\geq 3.$ Suppose that $V: \R^d \to \C^{N\times N}$ is such that $V \in C^\infty(\R^d\setminus\{x_0\})$ and that $V$ is bounded outside any ball containing $x_0$. If $\varepsilon_1, \varepsilon_2, \varepsilon_3,\varepsilon_4 \geq 0$ satisfy~\eqref{eq: hipotesis pequeñez} and $B$ and $V$ satisfy~\eqref{eq: B}-\eqref{eq: condition mass}, then $\sigma_p(H_m(A,V)) = \varnothing.$
	\end{corollary}
	Another easy consequence of Theorem~\ref{teo: general Dirac} is that as soon as this theorem is valid for an electric potential $V$ it is also valid for any conjugation of $V$ by unitary matrices.
	\begin{corollary}\label{cor:2}
		Let $d\geq 3.$ Suppose that $V$ is as in Theorem~\ref{teo: general Dirac}. If $\varepsilon_1, \varepsilon_2, \varepsilon_3,\varepsilon_4 $ satisfy~\eqref{eq: hipotesis pequeñez} and $B$ and $V$ satisfy~\eqref{eq: B}-\eqref{eq: condition mass}, then for any unitary matrix $P\in \C^{N\times N}$ one has $\sigma_p(H_m(A,PVP^\ast)) = \varnothing$. 
	\end{corollary}
	
	\begin{remark}[Decay behavior of $B$ and $V$]
		\label{remark: smallness}
		We observe that the Hardy-type conditions~\eqref{eq: B},~\eqref{eq: hipotesis 1} and~\eqref{eq: condition mass} prescribe a behavior at infinity of the magnetic field $B$ and the electric field $V:$  By virtue of the standard Hardy-inequality
		\begin{equation}\label{eq:Hardy-ineq}
			\int \frac{|\psi|^2}{|x|^2}\leq \frac{4}{(d-2)^2} \int |\nabla \psi|^2, \qquad \forall\, \psi\in H^1(\R^d),
		\end{equation}
		if $W$ satisfies the pointwise condition $|W|\leq \mu |x|^{-2}$ for some $\mu>0,$ then
		\begin{equation}\label{eq:Hardy-type} 
			\int |W||\psi|^2 
			\leq \frac{4\mu}{(d-2)^2}\int |\nabla \psi|^2
			=:c_{\mu,d} \int |\nabla \psi|^2.
		\end{equation}
		This provides a sufficient condition for assumptions~\eqref{eq: B},~\eqref{eq: hipotesis 1} and~\eqref{eq: condition mass} to be satisfied, \emph{i.e.} $|W|\approx \mu |x|^{-2},$ for some suitably small coupling constant $\mu>0.$  Notice that in the specific case of our conditions, namely~\eqref{eq: B},~\eqref{eq: hipotesis 1} and~\eqref{eq: condition mass}, one has $|W|=|x|^2|B|^2,$ $|W|=|x|^2|\nabla V|^2,$ and $|W|=|x||\nabla V|,$ respectively.
		
		On the other hand, one can prove (see~\cite[Thm. 2.1]{Ghoussoub2011} and~\cite{Sugie2002}) that if $W$ satisfies a Hardy-type inequality~\eqref{eq:Hardy-type}, then there exist $c, R > 0$ such that 
		\begin{equation*} 
			|W(x)| \leq \frac{c}{|x|^2}, \quad |x| > R.
		\end{equation*}
		So the $|x|^{-2}$ decay behavior is also a necessary condition for the validity of~\eqref{eq: B},~\eqref{eq: hipotesis 1} and~\eqref{eq: condition mass}.
		
		Finally, observe that, as a consequence of the diamagnetic inequality (see, for instance, \cite{Lieb2001})
		\begin{equation*}
			\left|\nabla|\psi(x)| \right| \leq |\nablaA \psi(x)|,
		\end{equation*}
		it is enough to check the validity of~\eqref{eq: B},~\eqref{eq: hipotesis 1} and~\eqref{eq: condition mass} above with the magnetic gradient replaced by the standard gradient.
	\end{remark}
	
	\begin{remark}[Refined version of condition~\eqref{eq: condition mass}]
		As one can see tracing the proof of Theorem~\ref{teo: general Dirac}, condition~\eqref{eq: condition mass} can be relaxed in the following way depending on whether we are interested in the behavior close to the origin or at infinity:
		\begin{equation}\label{eq:alternative-mass}
			\int \langle x\rangle |\nabla V||\psi|^2\leq \varepsilon_4^2 \|\nablaA \psi\|^2, 
		\end{equation}
		here $\langle x\rangle:=(1+|x|^2)^{1/2}$ is the standard notation for Japanese brackets. The alternative condition~\eqref{eq:alternative-mass} shows that in the massive case, or if $\{\beta, V\} \neq 0$, we can indeed allow for Coulomb-type decay near the origin, but not at infinity.  	
	\end{remark}
	
	\begin{remark}[Coulomb-type potentials]\label{remark:Coulomb-type}
		Notice that in the massless case, \emph{i.e.} $m=0,$ since~\eqref{eq: condition mass} is no longer needed, Theorem~\ref{teo: general Dirac} also covers Coulomb-type potentials, that is potentials of the form
		\begin{equation}\label{eq:Coulomb-type-pot}
			V(x)=\frac{\nu}{|x|} \mathbb{V},
		\end{equation}
		for sufficiently small $\nu$ (regardless of its sign) and where $\mathbb{V}$ is any fixed, constant-coefficients, Hermitian matrix.
		In particular, the following special case
		\begin{equation}\label{eq:V-special}
			\mathbb{V}=\mathbb{V}_{\text{el}}+\mathbb{V}_{\text{sc}}+\mathbb{V}_{\text{am}}
			:=\nu \I + \mu \beta+ i \delta \beta \Big(\boldsymbol{\alpha}\cdot \frac{x}{|x|} \Big)
		\end{equation}
		is included, where the potentials $\mathbb{V}_{\text{el}},\mathbb{V}_{\text{sc}}, \mathbb{V}_{\text{am}}$ are customarily called \emph{electrostatic}, (Lorentz-)\emph{scalar} and \emph{anomalous magnetic}, respectively and which are known to leave invariant the partial wave subspaces.
		
		\medskip
		We emphasize that the impossibility to cover Coulomb-type potentials~\eqref{eq:Coulomb-type-pot} in the massive case is not just technical. Indeed, it is known that, as soon as $m\neq 0,$ there are Coulomb-type perturbations that do produce discrete eigenvalues in the spectral gap $(-m,m)$ (see~\cite{Dolbeault2000} and references therein for the case $A=0$ and $\mathbb{V}_{\text{el}}$ and~\cite{Soff1973} for the case $\mathbb{V}_{\text{sc}}$). The situation is different for anomalous magnetic potentials $\mathbb{V}_{\text{am}}.$ In this case, since $\{\beta, \mathbb{V}_{\text{am}}\}=0,$ assumption~\eqref{eq: condition mass} is no longer needed. Thus, Theorem~\ref{teo: general Dirac} is valid for $V$ as in~\eqref{eq:Coulomb-type-pot} with $\mathbb{V}=\mathbb{V}_{\text{am}}$ also in the massive case, namely for any $m\geq 0.$ We emphasize that this does not contradict the explicit eigenvalues and eigenfunctions obtained in~\cite{Cassano2018} for Dirac operators with purely anomalous magnetic perturbations. Indeed, as can be seen from the proof of their result, if the coupling constant $\delta$ (there called $\lambda$) is sufficiently small, then the corresponding eigenfunctions found there are not $L^2$-integrable.
	\end{remark}
	We are now interested in stating the corresponding results to Theorem~\ref{teo: general Dirac} for the following potential
	\begin{equation*}
		V=V_{\text{el}}+V_{\text{sc}}+ V_{\text{am}}
		:=v_{\text{el}}(x) \I + v_{\text{sc}}(x) \beta + i\beta (\boldsymbol{\alpha} \cdot \nabla) \phi_{\text{am}}(x).
	\end{equation*}
	This is a generalization of the special case~\eqref{eq:Coulomb-type-pot} with~\eqref{eq:V-special}, and has been extensively studied as it is an interesting model in quantum mechanics (see, for example~\cite{Cassano2018} for further information).
	
	As we pointed out in Remark~\ref{remark:Coulomb-type}, Theorem~\ref{teo: general Dirac} does apply to these potentials if $V_\text{el}, V_\text{sc}$ and $V_\text{am}$ (or $v_\text{el}, v_\text{sc}$ and $\phi_\text{am}$) satisfy its hypotheses.  However, in these specific situations, stronger results can be obtained compared to the general case. This is why we will state and prove the corresponding results separately. 
	\begin{theorem}[Massive Dirac operator with electric potentials]
		\label{teo:electric}
		Let $d\geq 3.$ 
		Consider potentials of the form $V_{\text{el}}(x):=v_{\text{el}}(x) \I.$ Let $v_{\text{el}}\in L^1_{\text{loc}}(\R^d;\R)$ be such that $v_{\text{el}}\in C^\infty(\R^d\setminus \{0\})$ and that $v_{\text{el}}$ is bounded outside any ball containing the origin. Additionally, assume that there exist constants $\varepsilon_1, \varepsilon_2, \varepsilon_3, \varepsilon_4\geq 0$ satisfying
		\begin{equation}\label{eq:electric-hp-pequenez}
			\left(\frac{4d-6}{d-2}\right)\varepsilon_1+\left(\frac{4}{d-2}\right)\varepsilon_2\,\varepsilon_3 + \left(\frac{8d-8}{d-2}\right)\varepsilon_3 + 2\varepsilon_2+2m\varepsilon_4^2 < 2,
		\end{equation}
		and assume~\eqref{eq: B}-\eqref{eq: condition mass}. 
		Then $\sigma_p(H_m(A,V_\text{el})) = \varnothing$.  
	\end{theorem}	
	
	The previous result has as immediate corollary the analogous result in the massless case, namely when $m=0$. Nevertheless, in this case, different manipulation of the eigenvalue equation gives a stronger result than Theorem~\ref{teo:electric}.
	\begin{theorem}[Massless Dirac operator with electric potentials]
		\label{teo:electric-massles}
		Let $d\geq 3.$ 
		Consider potentials of the form $V_{\text{el}}(x):=v_{\text{el}}(x) \I.$ Let $v_{\text{el}}\in L^1_{\text{loc}}(\R^d;\R)$ be such that $v_{\text{el}}\in C^\infty(\R^d\setminus \{0\})$ and that $v_{\text{el}}$ is bounded outside any ball containing the origin. Additionally, assume that there exist constants $\varepsilon_1, \varepsilon_2\geq 0$ satisfying
		\begin{equation*}
			\Big(\frac{4d-7}{d-2} \Big) \varepsilon_1 +2\varepsilon_2+ \varepsilon_2^2<1
		\end{equation*}
		and assume~\eqref{eq: B} and~\eqref{eq: hipotesis 1}.
		Then $\sigma_p(H_0(A,V_\text{el})) = \varnothing$.  
	\end{theorem}

	\begin{theorem}[Scalar potentials]
		\label{teo:scalar}
		Let $d\geq 3.$ 
		Consider potentials of the form $V_{\text{sc}}(x):=v_{\text{sc}}(x) \beta.$ Let $v_{\text{sc}}\in L^1_{\text{loc}}(\R^d;\R)$ be such that $v_{\text{sc}}\in C^\infty(\R^d\setminus \{0\})$ and that $v_{\text{sc}}$ is bounded outside any ball containing the origin. Additionally, assume that there exist constants $\varepsilon_1, \varepsilon_2, \varepsilon_3, \varepsilon_4\geq 0$ satisfying
		\begin{equation}\label{eq:scalar-hp-pequenez}
			\left(\frac{4d-6}{d-2} \right)\varepsilon_1 
			+2\varepsilon_2\, \varepsilon_3 +2m \varepsilon_4^2		
			+\left(\frac{4d-4}{d-2} \right)\varepsilon_2
			< 2,
		\end{equation}
		and assume~\eqref{eq: B},~\eqref{eq: hipotesis 1} and, if $m\neq 0,$ assume also~\eqref{eq: condition mass}. Moreover,  assume that 
		\begin{equation}\label{eq:scalar-hipotesis}
			\int |v_{\text{sc}}|^2 |\psi|^2\leq \varepsilon_3^2 \|\nablaA \psi\|^2
		\end{equation}
		holds true for every $\psi \in H^1(\R^d)$. Then $\sigma_p(H_m(A,V_\text{sc})) = \varnothing$.  
	\end{theorem}

	\begin{theorem}[Anomalous magnetic potentials]
		\label{teo:anomalous-magnetic}
		Let $d\geq 3.$ 
		Consider potentials of the form $V_{\text{am}}(x):=i\beta (\boldsymbol{\alpha} \cdot \nabla) \phi_{\text{am}}(x).$ Let $\phi_{\text{am}}\in L^2_{\text{loc}}(\R^d;\R)$ be such that $\phi_{\text{am}}\in C^\infty(\R^d\setminus \{0\})$ and that $\phi_{\text{am}}$ is bounded outside any ball containing the origin. Additionally, assume that there exist constants $\varepsilon_1, \varepsilon_2, \varepsilon_3\geq 0$ satisfying
		\begin{equation}\label{eq:am-hp-pequenez}
			\left(\frac{4d-6}{d-2} \right) \varepsilon_1 
			+\left(\frac{16d-16}{d-2} \right)\varepsilon_2
			+\left(\frac{8d-8}{(d-2)^2} \right)\varepsilon_2^2 
			+\left(\frac{4d-4}{d-2} \right)\varepsilon_3
			< 2,
		\end{equation}
		and assume~\eqref{eq: B}. Moreover assume that
		\begin{align}
			\label{eq:am-hipotesis1}
			\sup_{x\in \R^d} |x||\nabla \phi_{\text{am}}|&\leq \varepsilon_2,\\
			\label{eq:am-hipotesis2}
			\int |x|^2|\Delta \phi_{\text{am}}|^2 |\psi|^2&\leq \varepsilon_3^2 \|\nablaA \psi\|^2,
		\end{align} 
		hold true for every $\psi \in H^1(\R^d)$. Then $\sigma_p(H_m(A,V_\text{am})) = \varnothing$.  
	\end{theorem}

	In three dimensions, we can consider an anomalous magnetic potential involving the spin operator, namely $V_{\text{am}}^{3d}:=i\beta (\boldsymbol{\alpha}\cdot \nabla)\phi_{\text{am}} -2 \beta \boldsymbol{S}\cdot B,$ where $\boldsymbol{S}=(S_1,S_2,S_3)$ represents the spin operator which is defined in terms of the Pauli matrices $\sigma_j,$ $j=1,2,3$ as follows
	\begin{equation*}
		S_j:= \frac{1}{2} 
		\begin{pmatrix}
			\sigma_j & 0\\
			0 & \sigma_j
		\end{pmatrix},
		\quad j=1,2,3.
	\end{equation*}
	\begin{theorem}[Anomalous magnetic, $d=3$]
		\label{teo:anomalous-magnetic-3d}
		Let $d= 3.$ 
		Consider potentials of the form $V_{\text{am}}^{3d}(x):=i\beta (\boldsymbol{\alpha} \cdot \nabla) \phi_{\text{am}}(x) -2\beta \boldsymbol{S}\cdot B(x).$ Let $\phi_{\text{am}}\in L^2_{\text{loc}}(\R^3;\R)$ be such that $\phi_{\text{am}}\in C^\infty(\R^3\setminus \{0\})$ and that $\phi_{\text{am}}$ is bounded outside any ball containing the origin. Additionally, assume that there exist constants $\varepsilon_1, \varepsilon_2, \varepsilon_3,\varepsilon_4, \varepsilon_5\geq 0$ satisfying
		\begin{multline}\label{eq:am-hp-pequenez-3d}
			\Big(\frac{4d-6}{d-2} \Big)\varepsilon_1 
			+\Big(\frac{8d-8}{(d-2)^2} \Big)\varepsilon_2^2
			+\Big(\frac{8d-8}{(d-2)^2} \Big)\varepsilon_4^2
			+\Big(\frac{16d-16}{(d-2)^2} \Big)\varepsilon_2\, \varepsilon_4
			+m\Big(\frac{16d-16}{d-2} \Big)\varepsilon_1\\
			+\Big(\frac{16d-16}{d-2} \Big)\varepsilon_2
			+\Big(\frac{24d-24}{d-2} \Big)\varepsilon_4 
			+\Big(\frac{4d-4}{d-2} \Big)\varepsilon_3
			+\Big(\frac{8d-8}{d-2} \Big)\varepsilon_5
			< 2,
		\end{multline}
		and assume~\eqref{eq: B},~\eqref{eq:am-hipotesis1} and~\eqref{eq:am-hipotesis2}. Moreover assume that
		\begin{align}
			\label{eq:am-B-3d}
			\sup_{x\in \R^d} |x||B(x)|&\leq \varepsilon_4,\\
			\label{eq:am-nablaB-3d}
			\int |x|^2|\nabla B|^2 |\psi|^2&\leq \varepsilon_5^2 \|\nablaA \psi\|^2,
		\end{align} 
		hold true for every $\psi \in H^1(\R^3)$. Then $\sigma_p(H_m(A,V_\text{am}^{3d})) = \varnothing$.  
	\end{theorem}

	\begin{remark}[Spectral stability]
		The results presented above provide sufficient conditions for the potentials to ensure absence of eigenvalues of the perturbed Hamiltonian~\eqref{eq:electromagnetic-Dirac}, which is crucial in some mathematical and quantum mechanical problems. In other words, we give sufficient conditions to ensure stability of the point spectrum under suitable perturbations. 
		
		In the purely electric case, namely if $A=0,$ one can actually prove stability of the entire spectrum. Indeed, condition~\eqref{eq: hipotesis sup} and Weyl's theorem ensure that $\sigma_{\text{ess}}(H_m(0,V)) =\sigma_{\text{ess}}(H_m)= (-\infty, m] \cup [m,\infty).$ Moreover, since the residual spectrum is always empty in the self-adjoint case, we finally have $\sigma(H_m(0,V))=\sigma(H_m).$ In order to obtain a similar result in the electromagnetic case, it would be enough to assume that $\lim_{|x|\to \infty} |A(x)|=0.$  This does not, in principle, follow from our conditions on $A,$ but this property does turn out to be true in many interesting models.
	\end{remark}
	
	In the previous results we have provided a source of sufficient conditions for guaranteeing \emph{total} absence of eigenvalues for $H_{A,V}.$ Nevertheless, the method of multipliers can be fruitfully used also to exclude \emph{embedded} eigenvalues only, as it is shown in the following result.
	\begin{theorem}\label{teo 2}
		Let $d \geq 3$. Suppose that $V: \R^d \to \C^{N \times N}$ is such that $V \in C^\infty(\R^d\setminus\{0\})$ and that $V$ is bounded outside any ball containing the origin. Assume also that there exist constants $\varepsilon_0, \varepsilon_1, \varepsilon_2, \varepsilon_3 \geq 0$ satisfying
		\begin{align}\label{hip-teo 2}
			\frac{(4d-5)}{2(d-2)}\left(\frac{\varepsilon_1}{\varepsilon_0}+\varepsilon_0\right)+\left(\frac{4}{d-2}\right)\varepsilon_2\,\varepsilon_3+2m\varepsilon_4^2+\left(\frac{8d-6}{d-2}\right)\varepsilon_3+\left(\frac{4d-2}{d-2}\right)\varepsilon_2+\frac{4}{(d-2)^2}(m^2+1)\varepsilon_3^2 \,\leq \,1,
		\end{align}
		such that $V$ satisfies \eqref{eq: hipotesis 1}-\eqref{eq: condition mass}. Additionally, suppose that the following inequality 
		\begin{align}
			\int |x|^2|B|^2|\psi|^2 &\leq \varepsilon_1\|\nablaA \psi\| ^2+\delta\|\psi\|^2 \label{eq: B2}
		\end{align} holds for all $\psi \in H^1(\R^d)$ and for some $\delta\geq 0.$ Then, the Dirac operator $H_{A,V}$ has no eigenvalues on $(-\infty, -\Lambda) \cup (\Lambda, \infty)$, where 
		\begin{align}\label{eq: Lambda}
			\Lambda = \sqrt{\frac{(4d-5)}{2(d-2)}\frac{\delta}{\varepsilon_0}+m^2+1}.
		\end{align}
	\end{theorem}
	
	\begin{remark}[Decay at infinity]
		Compared to hypothesis~\eqref{eq: B} in Theorem~\ref{teo: general Dirac}, hypothesis~\eqref{eq: B2} allows magnetic fields to decay slower. In fact, any $B$ such that $|B(x)| \leq \mu|x|^{-1}$ is admitted for $\mu$ small enough.
	\end{remark}
	\begin{remark}[Comparison with~\cite{Hundertmark2024}]\label{remark teo 2}
		 This result can be considered as a generalization of~\cite[Thm.8.1]{Hundertmark2024}, where \emph{purely magnetic} Dirac operators were considered. More specifically, there the authors showed that under similar hypothesis to~\eqref{eq: B2} for the magnetic field $B,$  the Dirac operator $H_{A,0}$ has no eigenvalues in $(-\infty, -\sqrt{\delta + m^2})\cup (\sqrt{\delta + m^2}, \infty).$ Additionally, they showed that if $B=o(|x|^{-1}),$ then $H_{A,0}$ has no eigenvalues in $(-\infty,-m)\cup (m, \infty).$ In our case, in the limit situation $\delta=0,$ we obtain that $H_{A,V}$ does not have eigenvalues in the smaller set
\begin{equation}\label{eq:translated}
\left(-\infty, -\sqrt{m^2+1}\right)\cup\left(\sqrt{m^2+1},\infty\right).
\end{equation}  
We emphasize that the translation by one in the threshold~\eqref{eq:translated} is due to the addition of the electric potential $V.$ More precisely, as one can explicitly see from the proof of Theorem~\ref{teo 2}, this comes from the need of estimating 
\begin{equation*}
	\langle \psi, m \{\beta, V \} \psi \rangle 
\end{equation*} 
for some suitable test function $\psi.$ Therefore, if $m \{\beta, V \}=0$ (in particular if $V=0$), then Theorem~\ref{teo 2} shows absence of eigenvalues in $(-\infty, -m)\cup (m,\infty),$ this, in particular, recovers the result in~\cite{Hundertmark2024}.
	\end{remark}
	\begin{remark}
		Notice that it is possible to remove condition~\eqref{eq: condition mass} from Theorem~\ref{teo 2}, which would allow to consider Coulomb-type potentials even in the massive case, under the cost of making smaller the interval where we can ensure the absence of eigenvalues. This is because in Theorem~\ref{teo: general Dirac} condition~\eqref{eq: condition mass}  is used to estimate the term \[ 
			\gen{2x\cdot \nabla_A \psi+d\psi, m\{\beta, V\}\psi}.
			 \]
And, while in Theorem~\ref{teo: general Dirac} we can not allow terms containing the $L^2$-norm of $\psi$, we can allow them in Theorem~\ref{teo 2} at a cost.
	\end{remark}
	\begin{remark}
		Finally, observe that the constant $\varepsilon_0$ in Theorem~\ref{teo 2} is not related to the smallness of the potentials. To convince oneself, it is enough to rewrite inequality~\eqref{hip-teo 2} as 
		\[ 
			\frac{(4d-5)}{2(d-2)}\frac{\varepsilon_1}{\varepsilon_0}+\left(\frac{4}{d-2}\right)\varepsilon_2\,\varepsilon_3+2m\varepsilon_4^2+\left(\frac{8d-6}{d-2}\right)\varepsilon_3+\left(\frac{4d-2}{d-2}\right)\varepsilon_2+\frac{4}{(d-2)^2}(m^2+1)\varepsilon_3^2 \,\leq \,1-	\frac{(4d-5)}{2(d-2)}\varepsilon_0,
		\]
		and to note that the only purpose of the constant $\varepsilon_0$ is to ensure that the right hand side is positive.
	\end{remark}

	\subsection{Organization of the paper.}
	The paper is organized as follows: In Section~\ref{section: definitions} we provide a precise definition of the electromagnetic Dirac operator~\eqref{eq:electromagnetic-Dirac}, giving also some minimal hypothesis on the potentials to ensure the self-adjointness of this operator in the appropriate domain. In Section~\ref{section: preliminary results} we provide some technical results needed to make rigorous all the computations. Finally, Section~\ref{section: proof of main theorem} is dedicated to the proofs of the results presented in the introduction.
	\subsection{Notation.}
	Here we compile some notation and useful properties used in this work.
	\begin{itemize}
		\item For a given magnetic potential $A: \R^d \to \R^d$ we define the magnetic gradient $\nablaA = \nabla-iA$, which satisfies the following product rule 
		\begin{equation}\label{eq:product-rule} 
			\nablaA (fg) = g\nablaA f+f\nabla g.
		\end{equation}
		We also define the magnetic field associated with $A$ as the matrix-valued function given by 
		\begin{equation*} 
			B_{jk}(x) = \partial_jA_k(x)-\partial_kA_j(x),
		\end{equation*}
		where the derivatives are taken in distributional sense. In dimension $d = 3$ we can identify $B$ with the unique vector $\hat{B} \in \R^3$ such that 
		\begin{equation*} 
			Bv = \hat{B} \times v
		\end{equation*}
		for all $v \in \R^3$, and in this case we are not going to make distinction between $B$ and $\hat{B}$.
		
		We refer to~\cite{Cazacu2016} 
		for a complete survey 
		on the concept of magnetic field in any dimensions and its definition in terms of differential forms and tensor fields.
		
		It is easy to see that the magnetic derivatives satisfy the following commutation relation
		\begin{equation}\label{eq: commutator}
			[\partial_j^A, \partial_k^A] = iB_{jk},
			\quad j,k\in \{1,2,\dots,d\}.
		\end{equation} 
		We will also use that
		\begin{equation} \label{eq: formula magnetic gradient}
			2\Re(\overline{\psi}\partial_j^A\psi) = 2\Re(\overline{\psi}\partial_j\psi) = \partial_j(|\psi|^2).
		\end{equation} 
		
		\item For a column vector $v \in \C^d$ its adjoint $v^*$ is the row vector given by 
		\begin{equation*} 
			(v^*)_j = \bar{v}_j.
		\end{equation*}
		With this, we have that the usual norm in $\C^d$ can be written as $|v|^2 = v^*v$. For matrices we use the same symbol for its norm, that is, if $M$ is a $d \times d$ matrix we set 
		\begin{equation*} 
			|M| = \sup_{v \in \C^d \setminus \{0\}} \frac{|Mv|}{|v|}.
		\end{equation*}
		\item For the sake of simplicity, we will write $L^2(\R^d)$ instead of $L^2(\R^d; \C^N)$ for vector valued functions. In the same spirit, the symbols $\| \cdot \|$ and $\gen{\cdot , \cdot}$ are used for the norm and scalar product in $L^2(\R^d),$ respectively.
		\item For two vectors $u,v$ we use the centered dot to denote the Euclidean scalar product, that is 
		\begin{equation*} 
			v \cdot u = \sum_{j = 1}^d v_ju_j.
		\end{equation*}
		This notation will be used regardless of the nature of $v_j$ and $u_j$. Thus, for example, if $\boldsymbol{\alpha} = (\alpha_1, \dots, \alpha_d)$ is a vector of matrices, we write
		\begin{equation*} 
			\begin{split}
				\boldsymbol{\alpha} \cdot \nabla  &=  \sum_{j = 1}^d \alpha_j \partial_j, \\
				\boldsymbol{\alpha} \cdot B \cdot \boldsymbol{\alpha} &= \sum_{j,k = 1}^dB_{jk}\alpha_j\alpha_k.
			\end{split}
		\end{equation*}
	\end{itemize}
	
	\section{Definition of Dirac operator}\label{section: definitions}
	In this section we gather some basic properties on the Dirac operators. We refer to~\cite{Thaller1992} for a more exhaustive dissertation.
	
	Following the notation used in the introduction, the free Dirac operator in dimension $d \geq 3$ is defined as 
	\begin{equation}\label{eq:free-Dirac}
		H_m \coloneqq H_m(0,0)
		= -i \boldsymbol{\alpha} \cdot \nabla+m\beta,
	\end{equation}
	where $\boldsymbol{\alpha} = (\alpha_1, \dots, \alpha_d)$ and $\beta$ are the so-called \textit{Dirac matrices}. These are $N \times N$ Hermitian matrices, with $N = 2^{\lfloor(d+1)/2 \rfloor}$, which satisfy the following anti-commutation relations 
	\begin{equation*}
		\begin{split}
			\alpha_j\alpha_k+\alpha_k\alpha_j &= 2\delta_{jk}\I_N, \\
			\alpha_j\beta +\beta\alpha_j &= 0,\\
			\beta^2 &= \I_N,
		\end{split}
	\end{equation*}
	for $j,k \in \{1,2,\dots , d\}$. For an explicit construction of these matrices in any dimension see for example~\cite{Cossetti2020}. It is well known that this operator is self-adjoint in $L^2(\R^d)$ with domain $D(H_m) = H^1(\R^d)$ and has pure continuous spectrum, namely 
	\begin{equation*}
		\sigma(H_m) = \sigma_c(H_m) = (-\infty, -m]\cup [m, \infty),
	\end{equation*}
	see for example \cite[Sec.1.4.4]{Thaller1992}. 
	
	In this work we are interested in perturbing~\eqref{eq:free-Dirac} with a magnetic potential $A: \R^d \to \R^d$ and an electric potential $V: \R^d \to \C^{N\times N}.$ More precisely, we are interested in the electromagnetic Dirac operator defined in~\eqref{eq:electromagnetic-Dirac}.
	
	In order to have that~\eqref{eq:electromagnetic-Dirac} is well defined as a self-adjoint operator in $L^2(\R^d),$ from now on, we will assume that $A \in L^2_{loc}(\R^d)$ is such that $B\in L^2_{loc}(\R^d)$ and such that $H_m(A):= H_m(A,0)$ is self-adjoint in $L^2(\R^d)$ with domain $H^1(\R^d)$. This class of potentials includes, for example, any smooth magnetic potential, or vector potentials that satisfy 
	\begin{align*}
		|A(x)| \leq \frac{a}{|x|}+b,
	\end{align*} 
	for $a>0$ small enough (see~\cite[Sec.4.3]{Thaller1992}). For the electric potential we shall assume that $V(x)$ is Hermitian for almost all $x \in \R^d$. Furthermore, to ensure the self-adjointness of $H_m(A,V)$ and, in consequence, that its spectrum is real, we will assume that $V$ is $H_m(A)$-bounded with relative bound less than 1. This means that there exist constants $a,b >0$ with $a \leq 1$ such that for every $\psi \in H^1(\R^d)$ the following inequality
	\begin{equation}\label{eq: HA bounded}
		\|V\psi\|^2 \leq a\|H_m(A)\psi\|^2+b\|\psi\|^2, 
	\end{equation} 
	holds. This allows us to define $H_m(A,V)$ as a self-adjoint operator in $L^2(\R^d)$ with domain $H^1(\R^d)$ by the Kato-Rellich theorem~\cite[Sec.X.2]{Reed1975}. Notice that in some specific cases of special interest, condition~\eqref{eq: HA bounded} can be replaced by weaker assumptions.
	
	In the next result we show that condition~\eqref{eq: HA bounded} implies that $V$ is $\nablaA$-bounded.
	\begin{lemma}\label{lemma: HA bounded 2}
		Let $V\colon\R^d \to \C^{N\times N}$ and $A\colon \R^d \to \R^d$ satisfy~\eqref{eq: HA bounded} and~\eqref{eq: B}, respectively. Then there exist constants $\mu,\nu > 0$ such that 
		\begin{equation}\label{eq: HA bounded 2}
			\|V\psi\|^2 \leq \mu \|\nablaA \psi\|^2+\nu\|\psi\|^2
		\end{equation}
		holds for all $\psi \in H^1(\R^d)$. 
	\end{lemma}
	\begin{proof}
		By density, we consider $\psi\in C^\infty_0(\R^d).$ First of all we observe that 
		\begin{equation*}
			\begin{split}
				\|H_m(A)\psi\|^2 &= \gen{\psi, H_m(A)^2\psi} \\
				& = \gen{\psi, -\Delta_{\text{A}} \psi }+\gen{\psi, -\tfrac{i}{2}(\boldsymbol{\alpha}\cdot  B \cdot \boldsymbol{\alpha}) \psi}+m^2\|\psi\|^2 \\
				&= \|\nablaA \psi\|^2+\gen{\psi, -\tfrac{i}{2}(\boldsymbol{\alpha}\cdot  B \cdot \boldsymbol{\alpha}) \psi}+m^2\|\psi\|^2.
			\end{split}
		\end{equation*}
		Using~\eqref{eq: B} and the Hardy inequality~\eqref{eq:Hardy-ineq} one has
		\begin{equation*}
			\begin{split}
				\gen{\psi, -\tfrac{i}{2}(\boldsymbol{\alpha}\cdot  B \cdot \boldsymbol{\alpha}) \psi} 
				&\leq \frac{1}{2}\int |B||\psi|^2 \\
				&\leq \frac{1}{2}\left(\int |x|^2|B|^2|\psi|^2\right)^{1/2}\left(\int \frac{|\psi|^2}{|x|^2}\right)^{1/2} \\
				&\leq \frac{\varepsilon_1}{d-2}\|\nablaA \psi\|^2.
			\end{split}
		\end{equation*}
		Plugging the latter inequality in the previous identity, we have
		\begin{equation}\label{eq:auxiliary}
			\|H_m(A)\psi\|^2 
			\leq \left(1+\frac{\varepsilon_1}{d-2}\right)\|\nablaA \psi\|^2+m^2\|\psi\|^2.
		\end{equation}
		Inserting this into~\eqref{eq: HA bounded} we obtain~\eqref{eq: HA bounded 2} with $\mu = a\big(1+\frac{\varepsilon_1}{d-2}\big)$ and $\nu = am^2+b.$
	\end{proof}
	
	
	
	\section{Preliminary results}\label{section: preliminary results}
	As mentioned in the introduction, the main ingredient in the proof of Theorem~\ref{teo: general Dirac} is the method of multipliers. The idea behind this method is to construct suitable weighted identities for solutions to the eigenvalue equation and then use smallness assumptions on the potentials to arrive to a contradiction.  To make the argument rigorous, we will need some preliminary lemmas contained in this section. 
	
	As a first step we need to approximate any solution $\psi$ of the eigenvalue equation with a sequence of compactly supported functions $\psi_R$ with compact support in $\R^d\setminus \{0\}.$ To achieve this, we define $\xi \in C_0^\infty(\R^d)$ such that 
	\begin{equation*} 
		\xi(x) = 
		\begin{cases}
			1 &\text{ if }|x| \leq 1, \\
			0 &\text{ if }|x| \geq 2 ,
		\end{cases}
	\end{equation*}
	and for $R>0$, we define  
	\begin{equation}\label{eq:xi_R}
		\xi_R(x) = \xi\left(\frac{x}{R}\right)-\xi(Rx).
	\end{equation}
	Then we have that 
	\begin{equation*}
		\nabla \xi_R(x) = \frac{1}{R}\nabla\xi\left(\frac{x}{R}\right)-R\nabla \xi(Rx),
		\quad \text{and} \quad
		\Delta \xi_R(x)= \frac{1}{R^2}\Delta \xi\left(\frac{x}{R}\right)-R^2\Delta \xi(Rx).
	\end{equation*}
	Notice that both $\nabla \xi (Rx)$ and $\Delta \xi(Rx)$ are supported in the anulus 
	\begin{equation}\label{eq:anulus} 
		D(R) = \left \{ x \in \R^d\colon \frac{1}{R} \leq |x| \leq \frac{2}{R} \right \},
	\end{equation}
	and, similarly, $\nabla \xi\left(\frac{x}{R}\right)$ and $\Delta \xi\left(\frac{x}{R}\right)$ are supported in $B(2R)$, the ball centered at 0 with radius $2R.$
	
	For any spinorial function $f:\R^d \to \C^N$ we define the compactly supported approximating family of functions with compact support in $\R^d\setminus \{0\}$ as $f_R = \xi_Rf$. 
	
	If $\psi \in H^1(\R^d)$ is a solution of 
	\begin{equation}\label{eq: eigenvalue}
		-i\boldsymbol{\alpha} \cdot \nablaA \psi +m\beta \psi+V\psi = \lambda \psi,
	\end{equation} 
	than it is easy to show that $\psi_R\in H^1(\R^d)$ and solves the approximating equation
	\begin{equation}\label{eq: aprox eigenvalue}
		-i \boldsymbol{\alpha}\cdot \nablaA \psi_R+m\beta \psi_R +V\psi_R = \lambda\psi_R-i(\boldsymbol{\alpha}\cdot \nabla\xi_R)\psi.
	\end{equation}
	
	The strategy to prove Theorem~\ref{teo: general Dirac} is to reduce to prove an analogous result for solutions to a second order eigenvalues equation. In the next result we will show that if $\psi_R\in H^1(\R^d)$ is a solution to the first order equation~\eqref{eq: aprox eigenvalue}, then it is more regular and therefore satisfies a second-order associated problem.  
	\begin{lemma}\label{lemma: aprox Pauli}
		Let $A$ and $V$ satisfy the hypothesis of Theorem \ref{teo: general Dirac}. If $\psi_R \in H^1(\R^d)$ is a solution of \eqref{eq: aprox eigenvalue}, then $\psi_R \in H^2(\R^d)$ and satisfies
		\begin{multline}\label{eq: aprox Pauli}
			H_m(A)^2\psi_R+V^2\psi_R+m\{\beta ,V\}\psi_R-i(\{\boldsymbol{\alpha}, V\}\cdot \nablaA)\psi_R-i(\boldsymbol{\alpha} \cdot \nabla V)\psi_R \\
			= \lambda^2\psi_R+(H_m(A,V)+\lambda)(-i(\boldsymbol{\alpha}\cdot \nabla)\xi_R)\psi.
		\end{multline}
	\end{lemma}
	\begin{proof}
		Notice that in order to prove $\psi_R \in H^2(\R^d)$ it is enough to show that $(\boldsymbol{\alpha} \cdot \nablaA)\psi_R \in H^1(\R^d)$. Since $\psi_R$ satisfies~\eqref{eq: aprox eigenvalue} then one needs to show that $V\psi_R\in H^1(\R^d)$ since the other terms in the identity are trivially in $H^1(\R^d).$
		Observe that since $\supp(\xi_R) \subseteq \{x \in \R^d : R^{-1} \leq |x| \leq 2R\},$ the same holds for $\psi_R$ by definition. Since $V$ is smooth and bounded in this set then $V\psi_R \in H^1(\R^d).$ 
		
		To prove \eqref{eq: aprox Pauli}, we apply $H_m(A,V)$ to both sides of \eqref{eq: aprox eigenvalue}. This is a rigorous step since we have just shown that all terms in the equation belong to $H^1(\R^d)$. 
		We consider first the left hand side.  
		\begin{equation*} 
			H_m(A,V)^2\psi_R = [H_m(A)^2+V^2+\{H_m(A),V\}]\psi_R.
		\end{equation*}
		Using the bi-linearity of the anti-commutator and the product rule for the magnetic gradient~\eqref{eq:product-rule} one has 
		\begin{equation*} 
			\{H_m(A), V\} = \{-i(\boldsymbol{\alpha} \cdot \nablaA), V\}+m\{\beta, V\}
			=-i(\{\boldsymbol{\alpha}, V\}\cdot \nablaA)-i(\boldsymbol{\alpha} \cdot \nabla) V +m\{\beta, V\}.
		\end{equation*}
		For the right hand side we have
		\begin{equation*}
			\begin{split}
				\lambda H_m(A,V)\psi_R +H_m(A,V)(-i(\boldsymbol{\alpha} \cdot \nabla\xi_R)\psi) &=\lambda (\lambda \psi_R-i(\boldsymbol{\alpha} \cdot \nabla\xi_R)\psi)
				+H_m(A,V)(-i(\boldsymbol{\alpha} \cdot \nabla\xi_R)\psi) \\
				&= \lambda^2\psi_R 
				+(H_m(A,V)+\lambda)(-i(\boldsymbol{\alpha} \cdot \nabla\xi_R)\psi),
			\end{split}
		\end{equation*}
		where we have simply used that $\psi_R$ satisfies~\eqref{eq: aprox eigenvalue}.
		Gathering the two identities gives the proof.
	\end{proof}
	The last term in the right-hand-side of~\eqref{eq: aprox Pauli} can be seen as an error term that comes from approximating $\psi$ by $\psi_R$. The following proposition shows that this error term becomes negligible as $R$ goes to infinity in a suitable sense dictated by the manipulation that we will make to~\eqref{eq: aprox Pauli}.
	\begin{prop}\label{prop:error-term-zero}
		Suppose that $V\colon \R^d\to \C^{N\times N}$ is such that $V\in \C^{\infty}(\R^d\setminus\{0\})$ and that $V$ is bounded outside any ball containing the origin. Moreover, assume~\eqref{eq: HA bounded}. 
		If $\psi\in H^1(\R^d)$ solves~\eqref{eq: eigenvalue} then 
		\begin{equation}\label{eq:error-term-zero}
			\lim_{R\to \infty} 
			\gen{2 (x\cdot \nablaA)\psi_R + d\psi_R, (H_m(A,V)+\lambda)(-i(\boldsymbol{\alpha}\cdot \nabla \xi_R)\psi)}
			=0.
		\end{equation}
	\end{prop}	
	In order to prove Proposition~\ref{prop:error-term-zero} we will need to show that  the $L^2$ norm of $(\nabla \xi_R)\psi$ becomes negligible as $R$ goes to $\infty.$  We will also need to show that $\psi_R$ converges to $\psi$ in $L^2(\R^d)$ and that $\nablaA \psi_R$ converges to $\nabla \psi.$ This is proved in the next lemma.
	\begin{lemma}\label{lemma: conv in H1}
		Let $\xi_R$ be as in~\eqref{eq:xi_R} above. Then for any $\psi \in H^1(\R^d)$ we have 
		\begin{align}
			\label{eq:lim1}
			&\lim_{R \to \infty} \|(\nabla\xi_R)\psi\| = 0, \\
			\label{eq:lim2}
			&\lim_{R \to \infty} \|\nablaA\psi-\nablaA\psi_R\|+\|\psi-\psi_R\| = 0.
		\end{align}
	\end{lemma}
	\begin{proof}
		Let us see~\eqref{eq:lim1} first. By definition, we have 
		\begin{equation*}
			(\nabla \xi_R) \psi= \frac{1}{R}\nabla\xi\left(\frac{x}{R} \right)\psi-R\nabla\xi(Rx)\psi.
		\end{equation*}
		Since $\nabla\xi$ is bounded, the first term converges to 0 in $L^2(\R^d)$. For the other term we use that  $\nabla \xi(Rx)$ is supported in the anulus $D(R)$ defined in~\eqref{eq:anulus}, and hence
		\begin{equation*}
			\begin{split}
				\|R\nabla \xi(Rx)\psi\|^2 &= \int_{D(R)}R^2|\nabla \xi(Rx)|^2|\psi|^2 \\
				& \leq 4\int_{D(R)} |\nabla \xi(Rx)|^2\frac{|\psi|^2}{|x|^2} \\
				& \leq 4\|\nabla \xi\|_{\infty}^2\int_{D(R)} \frac{|\psi|^2}{|x|^2}.
			\end{split}
		\end{equation*}
		Therefore, we only need to prove that the last integral converges to 0. From the Hardy inequality we have that $ \frac{|\psi|^2}{|x|^2}$ is integrable, and since the measure of $D(R)$ converges to $0$ as $R$ goes to infinity, we obtain the desired limit.
		
		To prove~\eqref{eq:lim2}, observe that from the definition of $\xi_R$ it follows that $\xi_R \to 1$ a.e. and $|\xi_R| \leq 2.$ Therefore, $\psi_R \to \psi$ in $L^2(\R^d)$ by the dominated convergence theorem. In order to show that $\nablaA \psi_R \to \nablaA \psi$ in $L^2(\R^d)$ one observes that 
		\begin{equation*} 
			\nablaA\psi_R = \xi_R \nablaA\psi+ (\nabla \xi_R) \psi.
		\end{equation*}
		Now the first term converges to $\nablaA\psi$ in $L^2(\R^d)$  using the same argument as above. For the second term, we have just shown that it converges to zero. This concludes~\eqref{eq:lim2} and the proof of the lemma.
	\end{proof}
	
	The proof of Proposition~\ref{prop:error-term-zero} immediately follows from the following lemma.
	\begin{lemma} \label{lemma: crucial lemma}
		Under the hypotheses of proposition~\ref{prop:error-term-zero} one has
		\begin{align}
			\label{eq:lim3}
			\lim_{R \to \infty} \gen{\psi_R, -i(\boldsymbol{\alpha}\cdot \nabla \xi_R)\psi} = 0,\\
			\label{eq:lim4}
			\lim_{R \to \infty} \gen{\psi_R, H_m(A,V)(-i(\boldsymbol{\alpha}\cdot \nabla \xi_R)\psi)} = 0,\\
			\label{eq: crucial 1}
			\lim_{R \to \infty} \langle 2x\cdot \nablaA \psi_R,H_m(A)(-i(\boldsymbol{\alpha}\cdot  \nabla \xi_R) \psi)\rangle = 0, \\
			\label{eq: crucial 2}
			\lim_{R \to \infty } \langle 2x\cdot \nablaA \psi_R,-i(\boldsymbol{\alpha} \cdot \nabla \xi_R)\psi \rangle = 0, \\
			\label{eq: crucial 3} 
			\lim_{R \to \infty} \langle 2x\cdot \nablaA \psi_R,V((-i\boldsymbol{\alpha} \cdot \nabla \xi_R)\psi)\rangle = 0. 
		\end{align}
	\end{lemma}
	\begin{proof}
		The first limit~\eqref{eq:lim3} simply follows from the Cauchy-Schwarz inequality and from~\eqref{eq:lim1} in Lemma~\ref{lemma: conv in H1}. 
		For the second identity~\eqref{eq:lim4} we observe  
		\begin{equation*}
			\begin{split}
				\gen{\psi_R, H_m(A,V)(-i(\boldsymbol{\alpha}\cdot \nabla \xi_R)\psi)} =& \gen{H_m(A)\psi_R, -i(\boldsymbol{\alpha}\cdot \nabla \xi_R)\psi}, 
				+\gen{V\psi_R, -i(\boldsymbol{\alpha}\cdot \nabla \xi_R)\psi} \\
				\leq & \|H_m(A)\psi_R\|\|(\nabla \xi_R)\psi\|+\|V\psi_R\|\|\|(\nabla \xi_R)\psi\|.
			\end{split}
		\end{equation*}
		From estimates~\eqref{eq: HA bounded 2} and~\eqref{eq:auxiliary} in Lemma~\ref{lemma: HA bounded 2}, we have that both $\|H_m(A)\psi_R\|$ and $\|V\psi_R\|$ are uniformly bounded in $R$ because, from Lemma~\ref{lemma: conv in H1}, $\psi_R$ converges in $H^1_A(\R^d)$. This allows us to conclude~\eqref{eq:lim4} since $(\nabla \xi_R)\psi$ converges to 0 in $L^2(\R^d)$ as it is shown in~\eqref{eq:lim1} of Lemma~\ref{lemma: conv in H1}.
		
		We continue proving~\eqref{eq: crucial 1}. One has
		\begin{equation*} 
			H_m(A)(-i(\boldsymbol{\alpha} \cdot \nabla \xi_R) \psi) = -\Delta \xi_R\psi-\alpha_j\alpha_k\partial_k\xi_R\partial_j^A\psi-im\beta (\boldsymbol{\alpha} \cdot \nabla \xi_R) \psi,
		\end{equation*}
		and then we have 
		\begin{multline}\label{eq:first-est}
			\langle 2x\cdot \nablaA \psi_R,H_m(A)(-i(\boldsymbol{\alpha}\cdot  \nabla) \xi_R \psi)\rangle\\ 
			= \langle 2x\cdot \nablaA \psi_R,-\Delta \xi_R\psi \rangle 
			+\langle 2x\cdot \nablaA\psi_R, -\alpha_j\alpha_k\partial_k\xi_R\partial_j^A\psi\rangle 
			+\gen{2x\cdot \nablaA \psi_R, -im\beta (\boldsymbol{\alpha} \cdot \nabla \xi_R)\psi}.
		\end{multline}
		We shall estimate each term of the right-hand-side of~\eqref{eq:first-est} separately. First,
		\begin{equation*}
			\begin{split}
				\left|\langle 2x\cdot \nablaA \psi_R,-\Delta \xi_R\psi \rangle \right|  & \leq 2 \int |x||\nablaA \psi_R| |\psi||\Delta \xi_R| \\
				& \leq  2\int_{D(R)} R^2 |x||\nablaA \psi_R| |\psi||\Delta \xi (Rx)| 
				+\frac{2}{R}\int_{B(2R)}\frac{|x|}{R}|\nablaA\psi_R||\psi| \left| \Delta \xi\left(\frac{x}{R}\right) \right| \\
				& \leq  8\|\Delta \xi\|_{\infty} \int_{D(R)}|\nablaA \psi_R| \frac{|\psi|}{|x|}
				+\frac{4\|\Delta \xi\|_{\infty}}{R}\int_{B(2R)}|\nablaA \psi_R||\psi|.
			\end{split}
		\end{equation*}
		From Lemma~\ref{lemma: conv in H1} we know that $\|\nablaA \psi_R\|$ is bounded. Moreover, since $\tfrac{|\psi|}{|x|}$ is square integrable, the first integral converges to 0 because the measure of $D(R)$ does. The convergence to zero of the second integral is immediate.
		
		Now we shall see the second term in the right-hand-side of~\eqref{eq:first-est}. Similarly as above one has
		\begin{equation*}
			\begin{split}
				|\langle 2x\cdot \nablaA\psi_R, -\alpha_j\alpha_k\partial_k\xi_R\partial_j^A\psi\rangle|  &\leq  2\int |x||\nablaA\psi_R||\nablaA \psi||\nabla \xi_R| \\
				& \leq 2\int_{D(R)}R|x||\nablaA\psi_R||\nablaA \psi|\left|\nabla\xi\left( Rx\right)\right| 
				+2\int_{B(2R)} \frac{|x|}{R}|\nablaA\psi_R||\nablaA \psi|\left|\nabla\xi\left(\frac{x}{R}\right)\right| \\
				&\leq  4\|\nabla \xi\|_{\infty}\int_{D(R)}|\nablaA\psi_R||\nablaA \psi| 
				+4 \int_{B(2R)}|\nablaA\psi_R||\nablaA \psi|\left|\nabla\xi\left(\frac{x}{R}\right)\right|. 
			\end{split}
		\end{equation*}
		The first integral converges to 0 in the same way as above. For the second we use that $\nabla\xi(\tfrac{x}{R})$ converges to 0 and is uniformly bounded in $R$. We conclude by the dominated convergence theorem. 
		
		For the last term in the right hand side of~\eqref{eq:first-est}, we have 
		\begin{equation*}
			\begin{split}
				|\gen{2x\cdot \nablaA \psi_R, -im\beta (\boldsymbol{\alpha} \cdot \nabla \xi_R)\psi}| &\leq   2m\int |x||\nablaA \psi_R| |\psi||\nabla \xi_R| \\
				&\leq  2m \int_{D(R)} R|x||\nablaA \psi_R||\psi| |\nabla \xi(Rx)| + 2m\int_{B(2R)} \frac{|x|}{R}|\nablaA \psi_R||\psi| \left|\nabla \xi\left(\frac{x}{R}\right)\right| \\
				&\leq  4m\|\nabla \xi\|_{\infty}\int_{D(R)}|\nablaA\psi_R| |\psi| + 4m \int_{B(2R)} |\nablaA \psi_R||\psi| \left|\nabla \xi\left(\frac{x}{R}\right)\right|.
			\end{split}
		\end{equation*}
		Hence, by using similar arguments as before, we conclude the proof of~\eqref{eq: crucial 1}. 
		
		Notice that the last estimate also provides the proof of~\eqref{eq: crucial 2}.
		
		Finally, to prove~\eqref{eq: crucial 3} one has
		\begin{equation*}
			\begin{split}
				|\langle 2x\cdot \nablaA \psi_R,V((-i\boldsymbol{\alpha} \cdot \nabla \xi_R)\psi)\rangle|
				&\leq   2\int |x||\nablaA \psi_R| |V||\psi||\nabla \xi_R| \\
				&\leq  2 \int_{D(R)} R|x||\nablaA \psi_R||V||\psi| |\nabla \xi(Rx)| + 2\int_{B(2R)} \frac{|x|}{R}|\nablaA \psi_R||V||\psi| \left|\nabla \xi\left(\frac{x}{R}\right)\right| \\
				&\leq  4\|\nabla \xi\|_{\infty}\int_{D(R)}|\nablaA\psi_R| |V||\psi| + 4 \int_{B(2R)} |\nablaA \psi_R||V||\psi| \left|\nabla \xi\left(\frac{x}{R}\right)\right|,
			\end{split}
		\end{equation*}
		where the convergence to zero in ensured reasoning as before using that $V$ satisfies~\eqref{eq: HA bounded 2}.
		This concludes the proof of the lemma.
	\end{proof}
	
	\begin{proof}[Proof of Proposition~\ref{prop:error-term-zero}]
		The proof of~\eqref{eq:error-term-zero} is an immediate consequence of the validity of~\eqref{eq:lim3}-\eqref{eq: crucial 3} in Lemma~\ref{lemma: crucial lemma}. 
	\end{proof}

	\begin{lemma}\label{lemma: multiplicador}
		Let $M\in W^{1,1}_{loc}(\R^d, \R^{N\times N})$ be a matrix-valued function which is Hermitian almost everywhere. Moreover assume $\psi \in H^2(\R^d)$ to be compactly supported and $c \in \R$. Then the following identities hold
		\begin{align}
			\label{eq:mult-1}
			\Re \gen{2x\cdot \nablaA \psi+d\psi, -\Delta_A\psi} 
			&= 2\|\nablaA \psi\|^2 +2\Im \int x_kB_{jk}\psi^*\partial_j^A\psi, \\
			\label{eq:mult-2}
			\Re\gen{2x\cdot \nabla \psi +d\psi,M\psi } &= -\gen{\psi, (x \cdot \nabla M) \psi}, \\
			\label{eq:mult-3}
			\Re\gen{2x\cdot \nabla \psi +d\psi,c\psi } &= 0.
		\end{align}
	\end{lemma}
	\begin{proof}
		We start by proving~\eqref{eq:mult-1}. Integrating by parts, we obtain
		\begin{equation*}
			\begin{split}
				\Re \gen{2x\cdot \nablaA \psi+d\psi, -\Delta_{\text{A}}\psi} 
				&= \Re\gen{2x\cdot \nablaA \psi, -\Delta_{\text{A}}\psi} -d\Re \gen{\psi, -\Delta_{\text{A}}\psi}\\
				&= \Re\gen{\nablaA(2x\cdot \nablaA \psi), \nablaA\psi} +d\|\nablaA\psi\|^2.
			\end{split}
		\end{equation*}
		By using the commutation relation~\eqref{eq: commutator} one easily gets 
		\begin{equation*}
			\begin{split}
				\Re\gen{\nablaA(2x\cdot \nablaA \psi), \nablaA\psi}
				&=\Re \gen{\partial_j^A(2x_k\partial_k^A \psi),\partial_j^A\psi} 
				= \Re \gen{2x_k\partial_k^A\partial_j^A \psi+2ix_kB_{jk}\psi+2\partial_j^A \psi,\partial_j^A \psi}\\
				&=(-d+2)\|\nablaA \psi\|^2 +2\Im \int x_kB_{jk}\psi^*\partial_j^A \psi, 
			\end{split}
		\end{equation*} 
		where we have used that 
		$2\Re (\partial_k^A \partial_j^A \psi)^* \partial_j^A \psi
		= \partial_k^A |\nablaA \psi|^2.$
		Notice that summation over repeated indices is implicit. 
		Plugging the latter in the former gives~\eqref{eq:mult-1}.
		
		Let us prove now~\eqref{eq:mult-2}. First, we compute $\Re\gen{2x\cdot \nabla \psi, M\psi}.$ Integrating by parts and using that $M$ is Hermitian, one has 
		\begin{equation*}
			\begin{split}
				\Re\gen{2x\cdot \nabla \psi, M\psi}
				&=-2d\gen{\psi,M\psi}
				-2\Re\gen{x_j\psi, \partial_j M \psi}
				-2\Re \gen{x_j\psi, M\partial_j \psi}\\
				&=-2d\gen{\psi,M\psi}
				-2\Re\gen{\psi, (x\cdot \nabla M) \psi}
				-\Re \gen{M\psi, 2x \cdot \nabla \psi}.
			\end{split}
		\end{equation*}
		This gives 
		\begin{equation*}
			\Re\gen{2x\cdot \nabla \psi, M\psi}=-d\gen{\psi,M\psi}
			-\Re\gen{\psi, (x\cdot \nabla M) \psi}.
		\end{equation*}
		And, identity~\eqref{eq:mult-2} follows immediately.
		We finally turn to~\eqref{eq:mult-3}, which trivially follows from~\eqref{eq:mult-2} observing that any constant $c\in \R$ belongs to $W^{1,1}_{loc}(\R^d).$
	\end{proof}

	\section{Proof of main theorems}\label{section: proof of main theorem}
	\begin{proof}[Proof of Theorem~\ref{teo 1}]
		By contradiction suppose that there exist $\lambda \in (-m,m)$ and $\psi \in H^1(\R^d)$ such that
		\begin{equation*} 
			H_m(A,V)\psi = \lambda \psi.
		\end{equation*}
		In particular, it follows that
		\begin{equation}\label{eq:preliminary-teo1} 
			\|H_m(A,V)\psi\|^2 = \lambda^2 \|\psi\|^2.
		\end{equation}
		Computing the left-hand side gives
		\begin{equation*}
			\|H_m(A,V)\psi\|^2 
			=\|H_m(A)\psi\|^2 + \|V\psi\|^2+2\Re\langle V\psi, H_m(A)\psi\rangle,
		\end{equation*}
		moreover 
		\begin{equation*}
			\begin{split}
				\|H_m(A)\psi\|^2 &= \|-i\boldsymbol{\alpha} \cdot \nablaA \psi\|^2
				+m^2\|\psi\|^2
				+2\Re\langle -i\boldsymbol{\alpha} \cdot \nablaA \psi, m\beta \psi\rangle \\
				&= \|-i\boldsymbol{\alpha} \cdot \nablaA \psi\|^2+m^2\| \psi\|^2,
			\end{split}
		\end{equation*}
		where in the last identity we have used that $\Re\langle -i\boldsymbol{\alpha} \cdot \nablaA \psi, m\beta \psi\rangle = 0$ since $-i (\boldsymbol{\alpha} \cdot \nablaA)$, and $\beta$ are anti-commuting self-adjoint operators. 
		
		Plugging the latter in the former gives 
		\begin{equation*} 
			\|H_m(A,V)\psi\|^2 = \|-i\boldsymbol{\alpha} \cdot \nablaA \psi\|^2+m^2\| \psi\|^2 + \|V\psi\|^2+2\Re\langle V\psi, H_m(A)\psi\rangle.
		\end{equation*} 
		Using this in~\eqref{eq:preliminary-teo1} we have
		\begin{equation*}
			\|-i\boldsymbol{\alpha} \cdot \nablaA \psi\|^2+\|V\psi\|^2 +2\Re\langle V\psi, H_m(A)\psi\rangle = (\lambda^2-m^2)\|\psi\|^2.
		\end{equation*}
		From the Cauchy-Schwartz inequality, we see that the left hand side of the equation above is bounded from below by
		\begin{align*}
			\|-i\boldsymbol{\alpha} \cdot \nablaA \psi\|^2+\|V\psi\|^2 -2 \| V\psi \| \| H_m(A)\psi\| & =  \left(\|H_m(A)\psi\|-\|V\psi\|\right)^2 
			+ \|-i\boldsymbol{\alpha} \cdot \nablaA \psi\|^2-\|H_m(A)\psi\|^2 \\
			&= (\|H_m(A)\psi\|-\|V\psi\|)^2 - m^2\|\psi\|^2.
		\end{align*}
		The last term is positive by the assumption~\eqref{eq: hipo teo 1} on $V$. This implies
		\begin{equation*}
			(\lambda^2-m^2)\|\psi\|^2\geq 0.
		\end{equation*} 
		Since $\lambda\in (-m,m),$ we have $\lambda^2-m^2<0.$  Therefore, the previous inequality implies that $\psi = 0,$ which is a contradiction.
	\end{proof}
	
	\begin{proof}[Proof of Theorem~\ref{teo: general Dirac}]
		Assume, by contradiction, that $\psi \in H^1(\R^d)$ is a non trivial solution of the eigenvalue problem~\eqref{eq: eigenvalue}. From Lemma~\ref{lemma: aprox Pauli} we have that the approximated solution $\psi_R: = \xi_R\psi$ satisfies~\eqref{eq: aprox Pauli}. Multiplying~\eqref{eq: aprox Pauli} by $2(x\cdot \nablaA)\psi_R+d\psi_R,$ integrating over $\R^d$ and taking the real part of the corresponding identity, we obtain
		\begin{multline}\label{eq: crucial identity}
			2\|\nablaA \psi_R\|^2+2\Im\int x_kB_{jk}\psi_R^*\partial_j^A\psi_R+\Re\gen{2x\cdot \nablaA\psi_R+d\psi_R,-\tfrac{i}{2}(\boldsymbol{\alpha} \cdot B \cdot \boldsymbol{\alpha}) \psi_R} \\
			-\gen{\psi_R, (x \cdot \nabla)(V^2)\psi_R }
			-m\gen{\psi_R, (x\cdot \nabla\{\beta, V\})\psi_R} \\
			+\Re\gen{2x\cdot \nablaA\psi_R+d\psi_R, -i(\{\boldsymbol{\alpha}, V\}\cdot \nablaA)\psi_R} 
			+\Re\gen{2x\cdot \nablaA\psi_R+d\psi_R,-i(\boldsymbol{\alpha} \cdot \nabla V)\psi_R} \\
			= \Re\gen{2x\cdot \nablaA\psi_R+d\psi_R, (H_m(A,V)+\lambda-i(\boldsymbol{\alpha} \cdot \nabla \xi_R) )\psi}.
		\end{multline}
		Here, we have used Lemma~\ref{lemma: multiplicador} and that $V^2$ and $\{\beta, V\}$ are Hermitian.
		We start estimating the left hand side of \eqref{eq: crucial identity}. We consider first the terms depending on $B.$ Using the Cauchy-Schwarz inequality and assumption~\eqref{eq: B} one has
		\begin{equation*}
			2\left|\Im\int x_kB_{jk}\psi_R^*\partial_j^A\psi_R\right| \leq 2\int|x||B||\psi_R||\nablaA\psi_R|
			\leq 2\varepsilon_1\|\nablaA\psi_R\|^2.
		\end{equation*}
		In a similar way, using that the Dirac matrices are unitary and the standard Hardy inequality, we have 
		\begin{equation*}
			\begin{split}
				\left|\Re\langle 2x\cdot \nablaA\psi_R+d\psi_R, -\tfrac{i}{2}(\boldsymbol{\alpha} \cdot B \cdot \boldsymbol{\alpha}) \psi_R \rangle\right| \leq& \int |x|||B||\psi_R|\nablaA\psi_R| +\frac{d}{2}\int|B||\psi_R|^2\\ 
				\leq &\varepsilon_1 \|\nablaA\psi_R\|^2+\frac{d}{2}\int|x||B||\psi_R| \frac{|\psi_R|}{|x|} \\
				\leq & \varepsilon_1 \|\nablaA\psi_R\|^2+\left(\frac{d}{d-2}\right)\varepsilon_1\|\nablaA\psi_R\|^2 \\
				\leq & \left(\frac{2d-2}{d-2}\right)\varepsilon_1\|\nablaA \psi_R\|^2.
			\end{split}
		\end{equation*}
		Similarly, for the terms involving the electric field $V,$ using~\eqref{eq: hipotesis 1} and~\eqref{eq: hipotesis sup}, one has 
		\begin{equation*}
			\begin{split}
				\left|\gen{\psi_R, (x \cdot \nabla)(V^2)\psi_R } \right| &\leq 2\int |x||V||\nabla V||\psi_R|^2  \\
				& \leq 2\left(\int |x|^2|\nabla V|^2|\psi_R|^2\right)^{1/2}\left(\int |V|^2|\psi_R|^2\right)^{1/2} \\
				& \leq \left(\frac{4}{d-2}\right) \varepsilon_2\,\varepsilon_3 \|\nablaA \psi_R\|^2.
			\end{split}
		\end{equation*}
		For the commutators of two matrices $U,V$, we will use that $|\{U,V\}| \leq 2|U||V|.$ Using this and~\eqref{eq: condition mass} one has
		\begin{equation*}
			\left| m\gen{\psi_R, (x\cdot \nabla\{\beta, V\})\psi_R} \right| \leq  2m\int |x||\nabla V||\psi_R|^2 
			\leq 2m \varepsilon_4^2\|\nablaA \psi_R\|^2.
		\end{equation*}
		Moreover, using again~\eqref{eq: hipotesis sup}, we have
		\begin{equation*}
			\begin{split}
				|\gen{2x\cdot \nablaA\psi_R+d\psi_R, -i(\{\boldsymbol{\alpha}, V\}\cdot \nablaA)\psi_R}| &\leq  4\int |x||V||\nablaA \psi_R|^2
				+2d\int |V||\psi_R||\nablaA\psi_R| \\
				&\leq  4\varepsilon_3 \|\nablaA\psi_R\|^2 + \left(\frac{4d}{d-2}\right)\varepsilon_3\|\nablaA \psi_R\|^2 \\
				&\leq \left(\frac{8d-8}{d-2}\right) \varepsilon_3\|\nablaA \psi_R\|^2,
			\end{split}
		\end{equation*}
		where we have used~\eqref{eq: hipotesis 1}.
		We conclude estimating the last term of the left-hand-side of~\eqref{eq: crucial identity}
		\begin{equation*}
			\begin{split}
				|\gen{2x\cdot \nablaA\psi_R+d\psi_R,-i(\boldsymbol{\alpha} \cdot \nabla V)\psi_R}| &\leq  2 \int |x||\nabla V| |\psi_R| |\nablaA \psi_R|
				+ d\int |\nabla V||\psi_R|^2 \\
				& \leq 2\varepsilon_2\|\nablaA \psi_R\|^2 
				+\left(\frac{2d}{d-2}\right)\varepsilon_2\|\nablaA \psi_R\|^2 \\
				& \leq \left(\frac{4d-4}{d-2}\right)\varepsilon_2\|\nablaA \psi_R\|^2.
			\end{split}
		\end{equation*}
		Gathering all the previous estimates and inserting them in~\eqref{eq: crucial identity} we get
		\begin{multline}\label{eq: crucial inequality}
			\left(2-\left(\frac{4d-6}{d-2}\right)\varepsilon_1-\left(\frac{4}{d-2}\right)\varepsilon_2\,\varepsilon_3-2m\varepsilon_4^2-\left(\frac{8d-8}{d-2}\right)\varepsilon_3-\left(\frac{4d-4}{d-2}\right)\varepsilon_2\right)\|\nablaA\psi_R\|^2 \\
			\leq \Re\gen{2(x\cdot \nablaA)\psi_R+d\psi_R, (H_m(A,V)+\lambda-i(\boldsymbol{\alpha} \cdot \nabla )\xi_R )\psi}
		\end{multline}
		Using~\eqref{eq:error-term-zero} in Proposition~\ref{prop:error-term-zero} we have that the right hand side of \eqref{eq: crucial inequality} converges to 0, and in virtue of \eqref{eq: hipotesis pequeñez} we must have that $\psi$ is identically zero, which leads to a contradiction.
	\end{proof}
	Since the arguments needed to prove Theorem 1.8 are similar to the ones in Theorem \ref{teo: general Dirac}, we will present its proof next, for clarity's sake.
	\begin{proof}[Proof of Theorem \ref{teo 2}]
		As in the proof of Theorem \ref{teo: general Dirac}, we assume by contradiction that $\psi \in H^1(\R^d)$ is a nontrivial solution of the eigenvalue equation \eqref{eq: eigenvalue}, for which we can derive in the same way as before the equation \eqref{eq: crucial identity}. 
		Moreover, in order to have an estimate involving the eigenvalue $\lambda$, we need, in this case, an additional identity: multiplying equation~\eqref{eq: aprox Pauli} by $\psi_R$ and integrating, then taking the real part of the resulting identity we get
	\begin{multline}\label{eq: multiplier 2} 
		\|\nabla_A \psi_R\|^2+\Re\gen{\psi_R, -\tfrac{i}{2}(\boldsymbol{\alpha}\cdot B \cdot \boldsymbol{\alpha}) \psi_R}+m^2\|\psi_R\|^2+\|V\psi_R\|^2 \\
		+\Re\gen{\psi_R, m\{\beta, V\}\psi_R}+\Re\gen{\psi_R, i(\{\boldsymbol{\alpha}, V\}\cdot \nablaA)\psi_R}-\Re\gen{ \psi_R,i(\boldsymbol{\alpha} \cdot \nabla V)\psi_R} \\
		= \lambda^2\|\psi_R\|^2 + \Re\gen{\psi_R, (H_m(A,V)+\lambda-i(\boldsymbol{\alpha} \cdot \nabla \xi_R) )\psi}.
	\end{multline}
	Now, taking the difference between~\eqref{eq: crucial identity} and~\eqref{eq: multiplier 2}, that is taking~\eqref{eq: crucial identity} -~\eqref{eq: multiplier 2} gives
	\begin{multline}\label{eq:difference}
		\|\nabla_A \psi_R\|^2 +2\Im\int x_kB_{jk}\psi_R^*\partial_j^A\psi_R+\Re\gen{2x\cdot \nablaA\psi_R+d\psi_R,-\tfrac{i}{2}(\boldsymbol{\alpha} \cdot B \cdot \boldsymbol{\alpha}) \psi_R} \\
			-\gen{\psi_R, (x \cdot \nabla)(V^2)\psi_R }
			-m\gen{\psi_R, (x\cdot \nabla\{\beta, V\})\psi_R} \\
			+\Re\gen{2x\cdot \nablaA\psi_R+d\psi_R, -i(\{\boldsymbol{\alpha}, V\}\cdot \nablaA)\psi_R} 
			+\Re\gen{2x\cdot \nablaA\psi_R+d\psi_R,-i(\boldsymbol{\alpha} \cdot \nabla V)\psi_R} \\
			-\Re\gen{\psi_R, -\tfrac{i}{2}(\boldsymbol{\alpha}\cdot B \cdot \boldsymbol{\alpha}) \psi_R}
			-m^2\|\psi_R\|^2
			-\|V\psi_R\|^2 \\
		-\Re\gen{\psi_R, m\{\beta, V\}\psi_R}
		-\Re\gen{\psi_R, i(\{\boldsymbol{\alpha}, V\}\cdot \nablaA)\psi_R}
		+\Re\gen{ \psi_R,i(\boldsymbol{\alpha} \cdot \nabla V)\psi_R}\\
			= -\lambda^2\|\psi_R\|^2 - \Re\gen{\psi_R, (H_m(A,V)+\lambda-i(\boldsymbol{\alpha} \cdot \nabla \xi_R) )\psi}\\ 
			+\Re\gen{2x\cdot \nablaA\psi_R+d\psi_R, (H_m(A,V)+\lambda-i(\boldsymbol{\alpha} \cdot \nabla \xi_R) )\psi}.
	\end{multline}
		Now we shall estimate each term of~\eqref{eq:difference}. Using assumption~\eqref{eq: B2} one has
		\begin{align*}
			2\left|\Im\int x_kB_{jk}\psi_R^*\partial_j^A\psi_R\right| &\leq 2\int|x||B||\psi_R||\nablaA\psi_R| \\
			& \leq \frac{1}{\varepsilon_0}\int|x|^2|B|^2|\psi_R|^2+\varepsilon_0 \| \nabla_A\psi_R\|^2\\
			& \leq \left(\frac{\varepsilon_1}{\varepsilon_0}+\varepsilon_0\right)\|\nabla_A\psi_R\|^2+\frac{\delta}{\varepsilon_0}\|\psi_R\|^2.
		\end{align*}
		For the second term in the first line of~\eqref{eq:difference} we have
		\begin{align*}
			\left|\Re\langle 2x\cdot \nablaA\psi_R+d\psi_R, -\tfrac{i}{2}(\boldsymbol{\alpha} \cdot B \cdot \boldsymbol{\alpha}) \psi_R \rangle\right| &\leq \int |x|||B||\psi_R|\nablaA\psi_R| +\frac{d}{2}\int|B||\psi_R|^2. 
		\end{align*}
		The first summand in the right-hand side is bounded in the same way as before. For the second, by using the Hardy inequality, we get
		\begin{align*}
			\frac{d}{2}\int |B||\psi_R|^2 &= \frac{d}{2}\int |x||B||\psi_R|\frac{|\psi_R|}{|x|} \\
			& \leq \frac{d}{(d-2)}\left(\int|x|^2|B|^2|\psi_R|^2\right)^{1/2}\|\nablaA\psi_R\| \\
			& \leq \frac{d}{2(d-2)}\left[\left(\frac{\varepsilon_1}{\varepsilon_0}+\varepsilon_0\right)\|\nabla_A  \psi_R\|^2+\frac{\delta}{\varepsilon_0}\|\psi_R\|^2 \right].
		\end{align*}
		Similarly, one can estimate the first term in the fourth line of~\eqref{eq:difference}. The terms in the second and third line of~\eqref{eq:difference} can be treated as in Theorem~\ref{teo: general Dirac}. Therefore it is left to estimates the remaining terms that depend on $V$ and that come from the introduction of the new identity~\eqref{eq: multiplier 2}.
		
		Using~\eqref{eq: hipotesis sup} and Hardy's inequality, we get
	\begin{align*}
		\|V\psi_R\|^2  \leq \varepsilon_3^2\int\frac{|\psi_R|^2}{|x|^2} \leq \frac{4\varepsilon_3^2}{(d-2)^2} \|\nabla_A\psi_R\|^2.
	\end{align*}
	Similarly
	\begin{align*}
		|\Re\gen{\psi_R, m\{\beta, V\}\psi_R}| & \leq 2m\int|V||\psi_R|^2 \\
		& \leq 4m\frac{\varepsilon_3}{d-2}\|\nabla_A\psi_R\|\|\psi_R\| \\
		& \leq 4m^2\frac{\varepsilon_3^2}{(d-2)^2}\|\nabla_A\psi_R\|^2+\|\psi_R\|^2.
	\end{align*}
	Again using~\eqref{eq: hipotesis sup}, one has
	\begin{align*}
		|\Re\gen{\psi_R, i(\{\boldsymbol{\alpha}, V\}\cdot \nablaA)\psi_R}| &\leq 2\int |V||\psi_R||\nabla_A\psi_R| \\
		& \leq \frac{4\varepsilon_3}{(d-2)}\|\nabla_A\psi_R\|^2.
	\end{align*}
	Finally
	\begin{align*}
		|\Re\gen{ \psi_R,i(\boldsymbol{\alpha} \cdot \nabla V)\psi_R}| & \leq \int |\nabla V||\psi_R|^2 \\
		& \leq \frac{2\varepsilon_2}{(d-2)}\|\nabla_A \psi_R\|^2.
	\end{align*}
	Using these estimates in~\eqref{eq:difference} and taking the limit as $R$ goes to infinity (using Lemma~\ref{prop:error-term-zero}) we obtain
%
	\begin{multline}\label{ineq final}
		\left(1
		-\frac{(4d-5)}{(d-2)}\left(\frac{\varepsilon_1}{\varepsilon_0}+\varepsilon_0\right)
		-\left(\frac{4}{d-2}\right)\varepsilon_2\,\varepsilon_3
		-2m\varepsilon_4^2
		-\left(\frac{8d-6}{d-2}\right)\varepsilon_3\right.\\
		\left.
		-\left(\frac{4d-2}{d-2}\right)\varepsilon_2
		-\frac{4}{(d-2)^2}\varepsilon_3^2 
		-\frac{4m^2}{(d-2)^2} \varepsilon_3^2
		\right)\|\nabla_A\psi\|^2
		+ 
		\left( 
		\lambda^2-m^2 -1 -\frac{\delta}{\varepsilon_0} \frac{(4d-5)}{2(d-2)}\right)\|\psi\|^2
		\leq 0.
	\end{multline}
	In virtue of \eqref{hip-teo 2} and \eqref{eq: Lambda} the left hand side of \eqref{ineq final} is non negative, and consequently $\psi = 0$, which gives us a contradiction. 
	\end{proof}

	\begin{proof}[Proof of Corollary~\ref{cor:1}]
		Theorem \ref{teo: general Dirac} remains valid if we move the singularity of the potentials from 0 to another point in $\R^d$. To prove this, observe that if $V$ is such that $V \in C^{\infty}(\R^d\setminus\{x_0\})$ and $\psi \in H^1(\R^d)$ is a solution of \[
		-i\boldsymbol{\alpha} \cdot \nablaA \psi+m\beta \psi + V\psi = \lambda \psi,
		\]
		then $\varphi(x) = \psi(x-x_0)$ is a solution of \[
		-i\boldsymbol{\alpha} \cdot \nabla_{\tilde{A}}\varphi++m\beta \varphi+\tilde{V}\varphi = \lambda\varphi,
		\]
		where $\tilde{A}(x) = A(x-x_0)$ and $\tilde{V}(x) = V(x-x_0)$. If $A$ and $V$ satisfy the hypotheses of the theorem, it is easy to check that $\tilde{A}$ and $\tilde{V}$ do too. 
	\end{proof}
	
	\begin{proof}[Proof of Corollary~\ref{cor:2}]
		Note also that if we prove that the Hamiltonian $-i\alpha \cdot \nablaA +m\beta +V$ has no eigenvalues, the same is true for any conjugation of $V$ by an unitary matrix. Indeed, if $\psi \in H^1(\R^d)$ is solution of \[
		-i\boldsymbol{\alpha} \cdot \nablaA\psi +m\beta\psi +V\psi = \lambda \psi
		\]
		and $P$ is an unitary matrix, then $\varphi = P\psi$ is solution of \[
		-i(P\boldsymbol{\alpha} P^*)\cdot \nablaA\varphi+m(P\beta P^*)\varphi+(PVP^*)\varphi = \lambda\varphi.
		\]
		Since conjugating the Dirac matrices produces another representation of them, the conjugated potential generates eigenvalues if and only if the original potential does.	
	\end{proof}
	
	\begin{proof}[Proof of Theorem~\ref{teo:electric}]
		Writing the identity~\eqref{eq: crucial identity} in the proof of Theorem~\ref{teo: general Dirac} for the specific case $V=V_{\text{el}}=v_{\text{el}}\I$, one has 
		\begin{multline}\label{eq:el-massive-final}
			2\|\nablaA \psi_R\|^2+2\Im\int x_kB_{jk}\psi_R^*\partial_j^A\psi_R+\Re\gen{2x\cdot \nablaA\psi_R+d\psi_R,-\tfrac{i}{2}(\boldsymbol{\alpha} \cdot B \cdot \boldsymbol{\alpha}) \psi_R}\\ 
			-\int x\cdot \nabla (v_\text{el}^2)|\psi_R|^2
			-2m\gen{\psi_R, (x\cdot \nabla v_{\text{el}}) \beta \psi_R }
			+\Re\gen{2x\cdot \nablaA\psi_R+d\psi_R, 2v_{\text{el}}(-i\boldsymbol{\alpha}\cdot \nablaA)\psi_R} \\
			+\Re\gen{2x\cdot \nablaA\psi_R+d\psi_R,-i(\boldsymbol{\alpha} \cdot \nabla v_{\text{el}})\psi_R} 
			= \Re\gen{2x\cdot \nablaA\psi_R+d\psi_R, (H_m(A,V_{\text{el}})+\lambda)-i(\boldsymbol{\alpha} \cdot \nabla \xi_R) \psi}.
		\end{multline}
		One can easily convince themselves that the only term that can be estimated differently from the general case is the following
		\begin{equation}\label{eq:I-special-form}
			\begin{split}
				I:&= \Re\gen{2x\cdot \nablaA\psi_R+d\psi_R,-i(\boldsymbol{\alpha} \cdot \nabla v_{\text{el}})\psi_R}\\
				&=\Re\gen{2x\cdot \nablaA\psi_R,-i(\boldsymbol{\alpha} \cdot \nabla v_{\text{el}})\psi_R}.
			\end{split}
		\end{equation}
		In the last identity we have used that since $-i\boldsymbol{\alpha}\cdot \nabla v_{\text{el}}$ is antysimmetric, $\Re \gen{\psi_R,-i\boldsymbol{\alpha}\cdot \nabla v_{\text{el}} \psi_R}=0.$ This gives us the simplified expression in~\eqref{eq:I-special-form}.
		Now, using the Cauchy-Schwarz inequality and hypothesis~\eqref{eq: hipotesis 1} one has
		\begin{equation*}
			\begin{split}
				|I|&\leq 2\int |x||\nabla v_{\text{el}}| |\psi_R| |\nablaA \psi_R|\\
				&\leq 2 \varepsilon_2\|\nablaA \psi_R\|^2.
			\end{split}	
		\end{equation*} 
		Note that, in the general case, assumption~\eqref{eq: hipotesis pequeñez} is also required to study $I.$ However, this is not necessary here, as $I$ turns out to have the special form~\eqref{eq:I-special-form}. 
		
		Using the bound above and estimating the other terms in~\eqref{eq:el-massive-final} as in the proof of Theorem~\ref{teo: general Dirac}, one has 
		\begin{multline*}
			\Bigg(2- \Big(\frac{4d-6}{d-2} \Big)\varepsilon_1 - \Big(\frac{4}{d-2} \Big)\varepsilon_2\, \varepsilon_3 - \Big(\frac{8d-8}{d-2} \Big)\varepsilon_3 -2\varepsilon_2 -4m \Big( \frac{2d-2}{d-2}\Big)\varepsilon_4 \Bigg) \|\nablaA \psi_R\|^2\\
			\leq \Re\gen{2x\cdot \nablaA\psi_R+d\psi_R, (H_m(A,V_{\text{el}})+\lambda)-i(\boldsymbol{\alpha} \cdot \nabla \xi_R) \psi}.
		\end{multline*}
		The result follows as in the proof of Theorem~\ref{teo: general Dirac}.
	\end{proof}	
	
	\begin{proof}[Proof of Theorem~\ref{teo:electric-massles}]
		Since $m=0,$ we can rewrite the eigenvalue equation~\eqref{eq: aprox Pauli} of Lemma~\ref{lemma: aprox Pauli} in a simplified way, namely
		\begin{equation*}
			H_0(A)^2\psi_R + 2v_{\text{el}} (-i \boldsymbol{\alpha} \cdot \nablaA + v_{\text{el}}) \psi_R 
			-v_{\text{el}}^2\psi_R
			-i (\boldsymbol{\alpha}\cdot \nabla v_{\text{el}})\psi_R
			= \lambda^2 \psi_R 
			+ (H_0(A, V_{\text{el}}) +\lambda) (-i \boldsymbol{\alpha} \cdot \nabla \xi_R) \psi.
		\end{equation*}
		In the second term of the previous identity we use that $\psi_R$ satisfies~\eqref{eq: aprox eigenvalue}, thus we get
		\begin{equation}\label{eq:electric-evs-eq}
			H_0(A)^2\psi_R 
			-i (\boldsymbol{\alpha}\cdot \nabla v_{\text{el}})\psi_R
			= (v_{\text{el}}-\lambda)^2 \psi_R 
			+ (H_0(A, -V_{\text{el}}) +\lambda) (-i \boldsymbol{\alpha} \cdot \nabla \xi_R) \psi.
		\end{equation}  
		Multiplying~\eqref{eq:electric-evs-eq} by $2(x\cdot \nablaA) \psi_R + d\psi_R,$ integrating over $\R^d$ and taking the real part of the resulting identity gives
		\begin{multline}\label{eq:el-first}
			2\|\nablaA \psi_R\|^2
			+2\Im\int x_kB_{jk}\psi_R^*\partial_j^A\psi_R
			+\Re\gen{2x\cdot \nablaA\psi_R+d\psi_R,-\tfrac{i}{2}(\boldsymbol{\alpha} \cdot B \cdot \boldsymbol{\alpha}) \psi_R} \\
			+\Re \gen{2x\cdot \nablaA\psi_R+d\psi_R, -i( \boldsymbol{\alpha} \cdot \nabla v_{\text{el}}) \psi_R}
			=-\int x\cdot \nabla \eta_\lambda |\psi_R|^2\\
			+\Re\gen{2x\cdot \nablaA\psi_R+d\psi_R, (H_0(A,-V_{\text{el}})+\lambda)(-i\boldsymbol{\alpha} \cdot \nabla \xi_R) \psi},
		\end{multline}
		where, to simplify the notation, we have defined $\eta_\lambda:=(v-\lambda)^2.$
		
		Now, multiplying~\eqref{eq:electric-evs-eq} by $\psi_R,$ integrating over $\R^d$ and taking the real part of the resulting identity gives
		\begin{equation}\label{eq:el-second}
			\begin{split}
				\|\nablaA \psi_R\|^2
				+\Re\gen{\psi_R,-&\tfrac{i}{2}(\boldsymbol{\alpha} \cdot B \cdot \boldsymbol{\alpha}) \psi_R} 
				+\Re \gen{\psi_R, -i( \boldsymbol{\alpha} \cdot \nabla v_{\text{el}}) \psi_R}\\
				&=\int \eta_\lambda |\psi_R|^2
				+\Re\gen{\psi_R, (H_0(A,-V_{\text{el}})+\lambda)(-i\boldsymbol{\alpha} \cdot \nabla \xi_R) \psi}.
			\end{split}
		\end{equation}
		Subtracting~\eqref{eq:el-second} from~\eqref{eq:el-first} gives
		\begin{multline}\label{eq:el-final}
			\|\nablaA \psi_R\|^2
			+2\Im\int x_kB_{jk}\psi_R^*\partial_j^A\psi_R
			+\Re\gen{2x\cdot \nablaA\psi_R+(d-1)\psi_R,-\tfrac{i}{2}(\boldsymbol{\alpha} \cdot B \cdot \boldsymbol{\alpha}) \psi_R} \\
			+\Re \gen{2x\cdot \nablaA\psi_R+(d-1)\psi_R, -i( \boldsymbol{\alpha} \cdot \nabla v_{\text{el}}) \psi_R}
			+ \int \eta_\lambda |\psi_R|^2 
			+\int x\cdot \nabla \eta_\lambda |\psi_R|^2\\
			=+\Re\gen{2x\cdot \nablaA\psi_R+(d-1)\psi_R, (H_0(A,-V_{\text{el}})+\lambda)(-i\boldsymbol{\alpha} \cdot \nabla \xi_R) \psi}.
		\end{multline}
		We only need to estimate the terms that depend on $\eta_\lambda.$ For the term containing $\nabla \eta_\lambda$ we proceed as follows:
		\begin{equation*}
			\begin{split}
				\Big|\int x \cdot \nabla \eta_\lambda |\psi_R|^2\Big|
				&\leq \Big( \int \frac{|x\cdot \nabla \eta_\lambda|^2}{\eta_\lambda} |\psi_R|^2 \Big)^{1/2} \Big( \int \eta_\lambda |\psi_R|^2 \Big)^{1/2}\\
				&\leq \Big( \int \frac{4|x\cdot \nabla v_{\text{el}}|^2(v-\lambda)^2}{(v-\lambda)^2} |\psi_R|^2 \Big)^{1/2} \Big( \int \eta_\lambda |\psi_R|^2 \Big)^{1/2}\\
				&\leq 2 \Big( \int |x|^2 |\nabla v_{\text{el}}|^2 |\psi_R|^2 \Big)^{1/2} \Big( \int \eta_\lambda |\psi_R|^2 \Big)^{1/2}\\
				&\leq \varepsilon_2^2 \|\nablaA \psi_R\|^2 + \int \eta_\lambda |\psi_R|^2, 
			\end{split}
		\end{equation*}
		where in the last inequality we have used~\eqref{eq: hipotesis 1}. Thus, one gets 
		\begin{equation*}
			\int \eta_\lambda |\psi_R|^2 
			+\int x\cdot \nabla \eta_\lambda |\psi_R|^2
			\geq - \varepsilon_2^2  \|\nablaA \psi_R\|^2.
		\end{equation*}
		Estimating the other terms in~\eqref{eq:el-final} as in the proof of Theorem~\ref{teo:electric} one gets
		\begin{equation*}
			\Big(1-\Big(\frac{4d-7}{d-2} \Big) \varepsilon_1 -2\varepsilon_2 -\varepsilon_2^2  \Big) \|\nablaA \psi_R\|^2
			\leq \Re\gen{2x\cdot \nablaA\psi_R+(d-1)\psi_R, (H_0(A,-V_{\text{el}})+\lambda)(-i\boldsymbol{\alpha} \cdot \nabla \xi_R) \psi}.
		\end{equation*}
		From here the result follows as in the proof of Theorem~\ref{teo: general Dirac}.
	\end{proof}
	
	\begin{proof}[Proof of Theorem~\ref{teo:scalar}]
		In this specific case, noticing that $\{\boldsymbol{\alpha}, V_{\text{sc}}\}=0,$ identity~\eqref{eq: aprox Pauli} in Lemma~\ref{lemma: aprox Pauli} reads as follows
		\begin{equation*}
			-\Delta_{\text{A}} \psi_R -\tfrac{i}{2} \boldsymbol{\alpha}\cdot B \cdot \boldsymbol{\alpha} \psi_R
			+ (m+ v_{\text{sc}})^2 \psi_R
			-i (\boldsymbol{\alpha} \cdot \nabla v_{\text{sc}})\beta \psi_R
			= \lambda^2 \psi_R + (H_m(A, V_{\text{sc}}) + \lambda) (-i\boldsymbol{\alpha}\cdot \nabla \xi_R \psi).
		\end{equation*}
		Multiplying the previous identity by $2x\cdot \nablaA\psi_R+d\psi_R,$ integrating over $\R^d$ and taking the real part one gets
		\begin{multline}\label{eq:scalar-last}
			2\|\nablaA \psi_R\|^2
			+2\Im\int x_kB_{jk}\psi_R^*\partial_j^A\psi_R
			+\Re\gen{2x\cdot \nablaA\psi_R+d\psi_R,-\tfrac{i}{2}(\boldsymbol{\alpha} \cdot B \cdot \boldsymbol{\alpha}) \psi_R} \\
			-\int x\cdot \nabla ((v_{\text{sc}} + m)^2) |\psi_R|^2
			+\Re \gen{2x\cdot \nablaA\psi_R+d\psi_R, -i( \boldsymbol{\alpha} \cdot \nabla v_{\text{sc}}) \beta \psi_R}\\
			=\Re\gen{2x\cdot \nablaA\psi_R+d\psi_R, (H_m(A,V_{\text{sc}})+\lambda)(-i\boldsymbol{\alpha} \cdot \nabla \xi_R) \psi},
		\end{multline}
		here we have used identity~\eqref{eq:mult-2} in Lemma~\ref{lemma: multiplicador}.
		We need to estimate only the first term in the second line of the previous identity as the other terms have been already estimated in the proof of the general result Theorem~\ref{teo: general Dirac}.
		We have 
		\begin{equation*}
			\begin{split}
				\Big|\int x \cdot \nabla ((v_{\text{sc}}+m)^2)|\psi_R|^2 \Big|&\leq 2\int |x||\nabla v_{\text{sc}}| |v_{\text{sc}} +m| |\psi_R|^2\\
				&\leq 2\int |x||\nabla v_{\text{sc}}| |v_{\text{sc}}||\psi_R|^2
				+ 2m \int |x||\nabla v_{\text{sc}}| |\psi_R|^2\\
				&\leq 2\Big(\int |x|^2|\nabla v_{\text{sc}}|^2 |\psi_R|^2\Big)^{1/2} \Big(|v_{\text{sc}}|^2|\psi_R|^2 \Big)^{1/2}
				2m \int |x||\nabla v_{\text{sc}}| |\psi_R|^2\\
				&\leq 2(\varepsilon_2 \, \varepsilon_3 + m\varepsilon_4^2 )\|\nablaA \psi_R\|^2,
			\end{split}
		\end{equation*}
		where in the last inequality we have used~\eqref{eq: hipotesis 1},~\eqref{eq: condition mass} and~\eqref{eq:scalar-hipotesis}.
		Using the bound above and estimating the other terms in~\eqref{eq:scalar-last} as in the proof of Theorem~\ref{teo: general Dirac} one gets
		\begin{multline*}
			\Bigg(2- \Big(\frac{4d-6}{d-2} \Big)\varepsilon_1 
			-2\varepsilon_2\, \varepsilon_3 -2m \varepsilon_4^2		
			- \Big(\frac{4d-4}{d-2} \Big)\varepsilon_2 \Bigg) \|\nablaA \psi_R\|^2\\
			\leq \Re\gen{2x\cdot \nablaA\psi_R+d\psi_R, (H_m(A,V_{\text{sc}})+\lambda)(-i\boldsymbol{\alpha} \cdot \nabla \xi_R) \psi}.
		\end{multline*}
		The result follows from the same arguments used in the proof of Theorem~\ref{teo: general Dirac}, using condition~\eqref{eq:scalar-hp-pequenez}.
	\end{proof}	
	
	Before continuing with the proof of Theorem~\ref{teo:anomalous-magnetic} we need a preliminary Lemma.
	\begin{lemma}\label{lemma:am-relations}
		Let $V_{\text{am}}=i\beta \boldsymbol{\alpha}\cdot \nabla\phi_{\text{am}}.$ The following relations hold true
		\begin{align}
			\label{eq:V-square}
			V_{am}^2&= |\nabla \phi_{\text{am}}|^2,\\
			\label{eq:V-beta}
			\{\beta, V_{\text{am}}\}&=0,\\
			\label{eq:V-alpha}
			\{\boldsymbol{\alpha}, V_{\text{am}}\} \nablaA&=2i\beta (\boldsymbol{\alpha}\cdot \nabla \phi_{\text{am}})(\boldsymbol{\alpha}\cdot \nablaA )
			+2i\beta \nabla \phi_{\text{am}} \cdot \nablaA\\
			\label{eq:nablaV-alpha}
			(\boldsymbol{\alpha}\cdot \nabla V_{\text{am}})&=-i\beta \Delta \phi_{\text{am}}.
		\end{align}
	\end{lemma}
	\begin{proof}
		The proof follows simply from the commutation relation properties of the Dirac matrices. For the sake of completeness we do the computations.
		We start by proving~\eqref{eq:V-square}. Note that
		\begin{equation*}
			V_{am}^2=-(\beta \alpha_j \partial_j \phi_{\text{am}})(\beta \alpha_k \partial_k \phi_{\text{am}})
			=(\alpha_j \partial_j \phi_{\text{am}})(\alpha_k \partial_k \phi_{\text{am}})
			= |\nabla \phi_{\text{am}}|^2.
		\end{equation*}
		As for~\eqref{eq:V-beta}, one has
		\begin{equation*}
			\{\beta, V_{\text{am}}\}
			=\beta(i\beta \alpha_j \partial_j\phi_{\text{am}}) + i\beta \alpha_j \beta \partial_j\phi_{\text{am}}
			=i(\alpha_j -\alpha_j)\partial_j \phi_{\text{am}}=0.
		\end{equation*}
		Regarding~\eqref{eq:V-alpha}, we have 
		\begin{equation*}
			\begin{split}
				\{\boldsymbol{\alpha}, V_{\text{am}}\} \nablaA
				&=\{\alpha_k, i \beta \alpha_j \partial_j \phi_{\text{am}}\}\partial_k^A 
				=i\beta (-\alpha_k \alpha_j + \alpha_j \alpha_k)\partial_j \phi_{\text{am}} \partial_k^A 
				=2i\beta (\alpha_j\alpha_k + \delta_jk)\partial_j \phi_{\text{am}} \partial_k^A \\
				&=2i\beta (\boldsymbol{\alpha}\cdot \nabla \phi_{\text{am}})(\boldsymbol{\alpha}\cdot \nablaA)
				+2i\beta \nabla \phi_{\text{am}} \cdot \nablaA.
			\end{split}
		\end{equation*}
		As for the last identity~\eqref{eq:nablaV-alpha}, one has 
		\begin{equation*}
			(\boldsymbol{\alpha}\cdot \nabla V_{\text{am}})
			=\alpha_j \partial_j(i\beta \boldsymbol{\alpha}\cdot \nabla \phi_{\text{am}})=-i\beta \alpha_j \alpha_k \partial_j \partial_k \phi_{\text{am}}=-i\beta \Delta \phi_{\text{am}}.
		\end{equation*}
		This concludes the proof of the Lemma.
	\end{proof}
	Now we are in position to prove Theorem~\ref{teo:anomalous-magnetic}.
	\begin{proof}[Proof of Theorem~\ref{teo:anomalous-magnetic}]
		In this specific case, using~\eqref{eq:V-square}-\eqref{eq:nablaV-alpha} in Lemma~\ref{lemma:am-relations}, the identity~\eqref{eq: aprox Pauli} in Lemma~\ref{lemma: aprox Pauli} reads as follows
		\begin{multline*}
			H_m(A)^2\psi_R + |\nabla \phi_{\text{am}}|^2 \psi_R
			-\beta (\Delta \phi_{\text{am}}) \psi_R
			+2\beta(\boldsymbol{\alpha}\cdot \nabla \phi_{\text{am}})(\boldsymbol{\alpha} \cdot \nablaA)\psi_R
			+2\beta \nabla \phi_{\text{am}} \cdot \nablaA \psi_R\\
			=\lambda^2 \psi_R + (H_m(A, V_{\text{am}})+\lambda)(-i\boldsymbol{\alpha} \cdot \nabla \xi_R \psi).
		\end{multline*}
		Multiplying the previous identity by $2x\cdot \nablaA \psi_R + d\psi_R,$ integrating over $\R^d$ and taking the real part we get
		\begin{multline}\label{eq:am-last}
			2\|\nablaA \psi_R\|^2
			+2\Im\int x_kB_{jk}\psi_R^*\partial_j^A\psi_R
			+\Re\gen{2x\cdot \nablaA\psi_R+d\psi_R,-\tfrac{i}{2}(\boldsymbol{\alpha} \cdot B \cdot \boldsymbol{\alpha}) \psi_R} \\
			+\Re\gen{2x\cdot \nablaA\psi_R+d\psi_R,|\nabla \phi_{\text{am}}|^2\psi_R}
			+\Re\gen{2x\cdot \nablaA\psi_R+d\psi_R, -\beta \Delta \phi_{\text{am}}\psi_R}\\
			+\Re\gen{2x\cdot \nablaA\psi_R+d\psi_R,2\beta (\boldsymbol{\alpha} \cdot \nabla \phi_{\text{am}})(\boldsymbol{\alpha} \cdot \nablaA \psi_R)}
			+\Re\gen{2x\cdot \nablaA\psi_R+d\psi_R,2\beta \nabla \phi_{\text{am}} \cdot \nablaA \psi_R)}\\
			=\Re\gen{2x\cdot \nablaA\psi_R+d\psi_R, (H_m(A,V_{\text{am}})+\lambda)(-i\boldsymbol{\alpha} \cdot \nabla \xi_R) \psi}.
		\end{multline}
		We start with the first term in the second line of~\eqref{eq:am-last}. Using~\eqref{eq:am-hipotesis1}, one has  
		\begin{equation*}
			\begin{split}
				|\Re\gen{2x\cdot \nablaA\psi_R+d\psi_R,|\nabla \phi_{\text{am}}|^2\psi_R}|
				&\leq 2\int |x||\nabla \phi_{\text{am}}|^2 |\psi_R||\nablaA \psi_R|
				+d \int |\nabla \phi_{\text{am}}|^2|\psi_R|^2\\
				&\leq 2\varepsilon_2^2 \int \frac{|\psi_R|}{|x|} |\nablaA \psi_R| + d \varepsilon_2^2 \int \frac{|\psi_R|^2}{|x|^2}\\
				&\leq \frac{4}{d-2} \varepsilon_2^2 \|\nablaA \psi_R\|^2 + \frac{4d}{(d-2)^2}\varepsilon_2^2 \|\nablaA \psi_R\|^2\\
				&=\frac{(8d-8)}{(d-2)^2}\varepsilon_2^2 \|\nablaA \psi_R\|^2.
			\end{split}
		\end{equation*}
		We continue with the second term in the second line of~\eqref{eq:am-last}, where using~\eqref{eq:am-hipotesis2}, we have
		\begin{equation*}
			\begin{split}
				|\Re\gen{2x\cdot \nablaA\psi_R+d\psi_R,-\beta \Delta \phi_{\text{am}}\psi_R}|
				&\leq 2\int |x||\Delta \phi_{\text{am}}| |\psi_R||\nablaA \psi_R|
				+d \int |\Delta \phi_{\text{am}}||\psi_R|^2\\
				&\leq 2\varepsilon_3 \|\nablaA \psi_R\|^2 + \frac{2d}{d-2} \varepsilon_3 \|\nablaA \psi_R\|^2 \\
				&\leq \frac{4d-4}{d-2} \varepsilon_3 \|\nablaA \psi_R\|^2.
			\end{split}
		\end{equation*}
		As for the first term in the third line of~\eqref{eq:am-last}, one has
		\begin{equation*}
			\begin{split}
				|\Re\gen{2x\cdot \nablaA\psi_R+d\psi_R,2\beta (\boldsymbol{\alpha}\cdot \nabla \phi_{\text{am}})(\boldsymbol{\alpha}\cdot \nablaA \psi_R)}|
				&\leq 4\int |x||\nabla \phi_{\text{am}}| |\nablaA \psi_R|^2
				+2d \int |\nabla \phi_{\text{am}}||\psi_R||\nablaA \psi_R|\\
				&\leq 4\varepsilon_2 \|\nablaA \psi_R\|^2 + \frac{4d}{d-2} \varepsilon_2 \|\nablaA \psi_R\|^2 \\
				&\leq \frac{8d-8}{d-2} \varepsilon_2 \|\nablaA \psi_R\|^2,
			\end{split}
		\end{equation*}
		where we have used again~\eqref{eq:am-hipotesis1}.
		The last term in the third line of~\eqref{eq:am-last} can be estimated as the previous one. 
		
		Plugging the previous estimates in~\eqref{eq:am-last} we get
		\begin{multline*}
			\Bigg(2- \Big(\frac{4d-6}{d-2} \Big)\varepsilon_1 
			- \Big(\frac{8d-8}{(d-2)^2} \Big)\varepsilon_2^2 
			-\Big(\frac{4d-4}{d-2} \Big)\varepsilon_3
			-2\Big(\frac{8d-8}{d-2} \Big)\varepsilon_2 \Bigg) \|\nablaA \psi_R\|^2\\
			\leq \Re\gen{2x\cdot \nablaA\psi_R+d\psi_R, (H_m(A,V_{\text{am}})+\lambda)(-i\boldsymbol{\alpha} \cdot \nabla \xi_R) \psi}.
		\end{multline*}
		Then the result follows as in the proof of Theorem~\ref{teo: general Dirac}, using condition~\eqref{eq:am-hp-pequenez}.
	\end{proof}
	
	Before proving Theorem~\ref{teo:anomalous-magnetic-3d} we need the following Lemma.
	\begin{lemma}\label{lemma:am-relations-3d}
		Let $V_{\text{am}}^{3d}=i\beta \boldsymbol{\alpha}\cdot \nabla\phi_{\text{am}}-2\beta \boldsymbol{S}\cdot B.$ The following relations hold true
		\begin{align}
			\label{eq:V-square-3d}
			(V_{am}^{3d})^2&= |\nabla \phi_{\text{am}}|^2 + |B|^2 +2\boldsymbol{\alpha} \cdot(\nabla \phi_{\text{am}} \times B),\\
			\label{eq:V-beta-3d}
			\{\beta, V_{\text{am}}^{3d}\}&=-4\boldsymbol{S}\cdot B,\\
			\label{eq:V-alpha-3d}
			\begin{split}
				\{\boldsymbol{\alpha}, V_{\text{am}}^{3d}\} \nablaA\psi_R&=2i\beta (\boldsymbol{\alpha}\cdot \nabla \phi_{\text{am}})(\boldsymbol{\alpha}\cdot \nablaA \psi_R)
				+2i\beta \nabla \phi_{\text{am}} \cdot \nablaA\psi_R\\
				&\phantom{=}-4\beta (\boldsymbol{S}\cdot B)(\boldsymbol{\alpha}\cdot \nablaA) +2\beta T(B\cdot \nablaA),
			\end{split}	
			\\
			\label{eq:nablaV-alpha-3d}
			(\boldsymbol{\alpha}\cdot \nabla V_{\text{am}}^{3d})&=
			-i\beta \Delta \phi_{\text{am}} + 2\beta \boldsymbol{\alpha} \cdot \nabla (\boldsymbol{S}\cdot B).
		\end{align}
	\end{lemma}
	\begin{proof}
		We will use the following notation: 
		\begin{equation*}
			V_{\text{am}}^{3d}=V_1+ V_2,
			\quad \text{where}\quad
			V_1:=i\beta \boldsymbol{\alpha}\cdot \nabla\phi_{\text{am}}
			\quad \text{and} \quad 
			V_2:=-2\beta \boldsymbol{S}\cdot B.
		\end{equation*}
		We start proving~\eqref{eq:V-square-3d}. Using~\eqref{eq:V-square} one has
		\begin{equation*}
			V_{am}^2=V_1^2 + V_2^2 + \{V_1,V_2\}
			= |\nabla \phi_{\text{am}}|^2 + V_2^2 + \{V_1,V_2\}.
		\end{equation*}
		For convenience, we introduce the $4\times 4$ matrix 
		\begin{equation*}
			T=
			\begin{pmatrix}
				0 & \I\\
				\I & 0
			\end{pmatrix}.
		\end{equation*}
		It is easy to see that for $j=1,2,3$ one has 
		\begin{equation}
			\label{eq:T-alpha-S-rel}
			2S_jT=2TS_j=\alpha_j, 
			\quad \text{or equivalently} \quad
			\alpha_j T=T\alpha_j=2S_j,
		\end{equation}
		and 
		\begin{equation}
			\label{eq:T-beta-rel}
			T\beta=-\beta T.
		\end{equation}
		With~\eqref{eq:T-alpha-S-rel} and~\eqref{eq:T-beta-rel} at hand we can now easily compute $V_2^2$ and $\{V_1,V_2\}.$ Note that
		\begin{equation*}
			V_2^2=(2\beta S_j B_j)(2\beta S_k B_k)
			=(\beta T\alpha_j  B_j)(\beta T \alpha_k B_k)
			=(T\alpha_j B_j)(T\alpha_k B_k)
			=(\boldsymbol{\alpha}\cdot B)(\boldsymbol{\alpha}\cdot B)
			= |B|^2.
		\end{equation*}
		Now,
		\begin{equation*}
			\begin{split}
				\{V_1,V_2\}
				&=\{i\beta \alpha_j \partial_j \phi_{\text{am}},-2\beta S_kB_k\}
				=iT(\boldsymbol{\alpha}\cdot\nabla \phi_{\text{am}}) (\boldsymbol{\alpha}\cdot B) 
				-iT(\boldsymbol{\alpha}\cdot B) (\boldsymbol{\alpha}\cdot \nabla \phi_{\text{am}})
				=4iT\boldsymbol{S}\cdot(\nabla \phi_{\text{am}} \times B)\\
				&=2\boldsymbol{\alpha}\cdot (\nabla \phi_{\text{am}} \times B),
			\end{split}
		\end{equation*}
		where we have used the algebraic identity $(\boldsymbol{\alpha}\cdot A)(\boldsymbol{\alpha}\cdot B)= A\cdot B + 2i \boldsymbol{S}\cdot (A \times B),$ that holds for any three-component vectors $A,B.$ Gathering these identities gives~\eqref{eq:V-square-3d}.
		As for~\eqref{eq:V-beta-3d} one has
		\begin{equation*}
			\{\beta, V_{\text{am}}^{3d}\}
			=\{\beta, V_1\} + \{\beta, V_2\}
			=0 +\{\beta, -2\beta S_k B_k\}
			=-4 \boldsymbol{S}\cdot B,
		\end{equation*}
		where we have used~\eqref{eq:V-square} and that for any $k=1,2,3$ one has $\beta S_k=S_k \beta$, which follows from~\eqref{eq:T-alpha-S-rel} and~\eqref{eq:T-beta-rel}.
		As for~\eqref{eq:V-alpha-3d} we have 
		\begin{equation*}
			\begin{split}
				\{\boldsymbol{\alpha}, V_{\text{am}}^{3d}\} \nablaA&=
				\{\boldsymbol{\alpha}, V_1\}\nablaA + \{\boldsymbol{\alpha}, V_2\}\nablaA
				= 	2i\beta (\boldsymbol{\alpha}\cdot \nabla \phi_{\text{am}})(\boldsymbol{\alpha}\cdot \nablaA )
				+2i\beta \nabla \phi_{\text{am}} \cdot \nablaA + \{\boldsymbol{\alpha}, V_2\}\nablaA\\
				&=2i\beta (\boldsymbol{\alpha}\cdot \nabla \phi_{\text{am}})(\boldsymbol{\alpha}\cdot \nablaA)
				+2i\beta \nabla \phi_{\text{am}} \cdot \nablaA
				-4\beta (\boldsymbol{S}\cdot B)(\boldsymbol{\alpha}\cdot \nablaA)
				+2\beta T(B\cdot \nablaA).
			\end{split}
		\end{equation*}
		Here we have used~\eqref{eq:V-alpha}.
		Regarding the last identity~\eqref{eq:nablaV-alpha-3d},
		\begin{equation*}
			(\boldsymbol{\alpha}\cdot \nabla V_{\text{am}}^{3d})
			=(\boldsymbol{\alpha}\cdot \nabla V_1) + (\boldsymbol{\alpha}\cdot \nabla V_2)
			=-i\beta \Delta \phi_{\text{am}} + (\boldsymbol{\alpha}\cdot \nabla V_2)
			=-i\beta \Delta \phi_{\text{am}} + 2\beta \boldsymbol{\alpha} \cdot \nabla (\boldsymbol{S}\cdot B).
		\end{equation*}
		where we have used~\eqref{eq:nablaV-alpha}.
		This concludes the proof of the Lemma.
	\end{proof}
	Now we are in position to prove Theorem~\ref{teo:anomalous-magnetic-3d}.
	\begin{proof}[Proof of Theorem~\ref{teo:anomalous-magnetic-3d}]
		In this specific case, using~\eqref{eq:V-square-3d}-\eqref{eq:nablaV-alpha-3d} in Lemma~\ref{lemma:am-relations-3d}, the identity~\eqref{eq: aprox Pauli} in Lemma~\ref{lemma: aprox Pauli} reads as follows
		\begin{multline*}
			H_m(A)^2\psi_R + |\nabla \phi_{\text{am}}|^2 \psi_R
			+|B|^2\psi_R 
			+2\boldsymbol{\alpha} \cdot (\nabla \phi_{\text{am}}\times B)\psi_R
			-4m\boldsymbol{S}\cdot B \psi_R\\
			+2\beta (\boldsymbol{\alpha}\cdot \nabla \phi_{\text{am}})(\boldsymbol{\alpha}\cdot \nablaA \psi_R)
			+2\beta (\nabla \phi_{\text{am}} \cdot \nablaA)\psi_R
			+4i\beta (\boldsymbol{S}\cdot B)(\boldsymbol{\alpha}\cdot \nablaA) \psi_R\\ 
			-2i\beta T(B\cdot \nablaA)\psi_R
			-\beta (\Delta \phi_{\text{am}}) \psi_R
			-2i\beta (\boldsymbol{\alpha}\cdot \nabla) (\boldsymbol{S}\cdot B) \psi_R\\
			=\lambda^2 \psi_R + (H_m(A, V_{\text{am}}^{3d})+\lambda)(-i\boldsymbol{\alpha} \cdot \nabla \xi_R \psi).
		\end{multline*}
		Multiplying the previous identity by $2x\cdot \nablaA \psi_R + d\psi_R,$ integrating over $\R^d$ and taking the real part we get
		\begin{multline*}
			2\|\nablaA \psi_R\|^2
			+2\Im\int x_kB_{jk}\psi_R^*\partial_j^A\psi_R
			+\Re\gen{2x\cdot \nablaA\psi_R+d\psi_R,-\tfrac{i}{2}(\boldsymbol{\alpha} \cdot B \cdot \boldsymbol{\alpha}) \psi_R} \\
			+\Re\gen{2x\cdot \nablaA\psi_R+d\psi_R,|\nabla \phi_{\text{am}}|^2\psi_R}
			+\Re\gen{2x\cdot \nablaA\psi_R+d\psi_R,|B|^2\psi_R}\\
			+\Re\gen{2x\cdot \nablaA\psi_R+d\psi_R, 2\boldsymbol{\alpha}\cdot (\nabla \phi_{\text{am}}\times B)\psi_R}
			-4m\Re\gen{2x\cdot \nablaA\psi_R+d\psi_R,\boldsymbol{S}\cdot B\psi_R}\\
			+\Re\gen{2x\cdot \nablaA\psi_R+d\psi_R,2\beta (\boldsymbol{\alpha} \cdot \nabla \phi_{\text{am}})(\boldsymbol{\alpha} \cdot \nablaA \psi_R)}
			+\Re\gen{2x\cdot \nablaA\psi_R+d\psi_R,2\beta (\nabla \phi_{\text{am}} \cdot \nablaA) \psi_R}\\
			+\Re\gen{2x\cdot \nablaA\psi_R+d\psi_R,4i\beta (\boldsymbol{S} \cdot B)(\boldsymbol{\alpha}\cdot \nablaA)\psi_R}
			+\Re\gen{2x\cdot \nablaA\psi_R+d\psi_R,-2i\beta T (B\cdot \nablaA) \psi_R}\\
			+\Re\gen{2x\cdot \nablaA\psi_R+d\psi_R, -\beta \Delta \phi_{\text{am}}\psi_R}
			+\Re\gen{2x\cdot \nablaA\psi_R+d\psi_R, -2i\beta (\boldsymbol{\alpha}\cdot \nabla) (\boldsymbol{S}\cdot B)\psi_R}\\
			=\Re\gen{2x\cdot \nablaA\psi_R+d\psi_R, (H_m(A,V_{\text{am}}^{3d})+\lambda)(-i\boldsymbol{\alpha} \cdot \nabla \xi_R) \psi}.
		\end{multline*}
		Proceeding similarly to the proof of Theorem~\ref{teo:anomalous-magnetic} and using~\eqref{eq: B},~\eqref{eq:am-hipotesis1},~\eqref{eq:am-hipotesis2},~\eqref{eq:am-B-3d} and~\eqref{eq:am-nablaB-3d}, we get
		\begin{multline*}
			\Bigg(2- \Big(\frac{4d-6}{d-2} \Big)\varepsilon_1 
			- \Big(\frac{8d-8}{(d-2)^2} \Big)\varepsilon_2^2
			- \Big(\frac{8d-8}{(d-2)^2} \Big)\varepsilon_4^2
			- \Big(\frac{16d-16}{(d-2)^2} \Big)\varepsilon_2\, \varepsilon_4
			-m\Big(\frac{16d-16}{d-2} \Big)\varepsilon_1\\
			-\Big(\frac{16d-16}{d-2} \Big)\varepsilon_2
			-\Big(\frac{24d-24}{d-2} \Big)\varepsilon_4 
			-\Big(\frac{4d-4}{d-2} \Big)\varepsilon_3
			-\Big(\frac{8d-8}{d-2} \Big)\varepsilon_5 \Bigg) \|\nablaA \psi_R\|^2\\
			\leq \Re\gen{2x\cdot \nablaA\psi_R+d\psi_R, (H_m(A,V_{\text{am}}^{3d})+\lambda)(-i\boldsymbol{\alpha} \cdot \nabla \xi_R) \psi}.
		\end{multline*}	
	\end{proof}
	
	\section*{Acknowledgements}
	N. A. was supported by the grants PID2021-126813NB-I00, funded by MCIN/AEI/10.13039/501100011033, and IT1615-22, funded by the Basque Government.
	
	L.C. was supported by the grant Ramón y Cajal RYC2021-032803-I funded by MCIN/AEI/10.13039/50110 and by Ikerbasque. 
	
	M.M was supported by the grant PID2021-126813NB-I00 funded by MCIN/AEI/10.13039/501100011033 and by the FSE+.
	
	\section*{Data availability}
	There is no data associated with this article.
	\section*{Conflict of interest}
	There is no conflict of interest to declare.
	\bibliography{Dirac.bib}
	\bibliographystyle{abbrv} 
\end{document}